\documentclass[twoside,11pt]{article}

\usepackage{jmlr2e}

\usepackage{amsmath}
\usepackage{amsfonts}

\usepackage{subcaption}
\usepackage[inline]{enumitem}

\usepackage{algorithm}
\usepackage{algorithmic}

\usepackage{color}
\usepackage[dvipsnames]{xcolor}
\usepackage{soul}

\newtheorem{assumption}{Assumption}

\def\as#1{{{\color{black}#1}}}

\def\E{\mathbb{E}}
\def\R{\mathbb{R}}
\def\1{{\bf 1}}

\def\bA{{\mathbf A}}
\def\bB{{\mathbf B}}

\def\S{\mathcal{S}}
\def\V{{\mathcal V}}
\def\G{{\mathcal G}}
\def\H{{\mathcal H}}
\def\Ed{{\mathcal E}}
\def\O{{\mathcal O}}

\def\bM{{\mathbf M}}

\def\bP{{\mathbf P}}

\def\bx{{\mathbf x}}

\def\bbw{\bar{\mathbf w}}
\def\bw{{\mathbf w}}
\def\bz{{\mathbf z}}
\def\bg{{\mathbf g}}
\def\bu{{\mathbf u}}
\def\bv{{\mathbf v}}
\def\by{{\mathbf y}}

\def\br{{\mathbf r}}
\def\be{\boldsymbol{\varepsilon}}
\def\bbe{\bar \be}
\def\hbg{\hat{\mathbf g}}

\def\bDelta{{\boldsymbol \Delta}}
\def\bchi{{\boldsymbol \chi}}
\def\bpsi{{\boldsymbol \psi}}
\def\bphi{{\boldsymbol \phi}}
\def\etab{{\boldsymbol \eta}}

\ShortHeadings{Robust Asynchronous Stochastic Gradient-Push}{Spiridonoff, Olshevsky and Paschalidis}
\firstpageno{1}

\begin{document}

\title{Robust Asynchronous Stochastic Gradient-Push:\\
	Asymptotically Optimal and Network-Independent Performance for Strongly Convex Functions}

\author{\name Artin Spiridonoff \email artin@bu.edu \\
       \addr Division of Systems Engineering\\
       Boston University\\
       Boston, MA 02215, USA
       \AND
       \name Alex Olshevsky \email alexols@bu.edu \\
       \addr Department Electrical and Computer Engineering\\
       Boston University\\
       Boston, MA 02215, USA
       \AND
       \name Ioannis Ch. Paschalidis \email yannisp@bu.edu \\
       \addr Department Electrical and Computer Engineering\\
       Boston University\\
       Boston, MA 02215, USA
       }
   
\maketitle

\begin{abstract}
We consider the standard model of distributed optimization of a  sum of functions $F(\bz) = \sum_{i=1}^n f_i(\bz)$, where node $i$ in a network holds the function $f_i(\bz)$. We allow for a harsh network model characterized by asynchronous updates, message delays, unpredictable message losses, and directed communication among nodes.  In this setting, we analyze a modification of the Gradient-Push method for distributed optimization, assuming that \begin{enumerate*}[label=(\roman*)] \item node $i$ is capable of generating gradients of its function $f_i(\bz)$ corrupted by zero-mean bounded-support additive noise at each step, \item $F(\bz)$ is strongly convex, and \item each $f_i(\bz)$ has Lipschitz gradients. We show that our proposed  method asymptotically performs  as well as the best bounds on centralized  gradient descent that takes steps in the direction of the  sum of the noisy gradients of all the functions $f_1(\bz), \ldots, f_n(\bz)$ at each step.
\end{enumerate*}
\end{abstract}


\section{Introduction}

Distributed systems have attracted much  attention in recent years due to their \as{many applications such as large scale machine learning (e.g., in the healthcare domain, \cite{brisimi2018federated}), control (e.g., maneuvering of autonomous vehicles, \cite{peng2017distributed}), sensor networks (e.g., coverage control, \cite{he2015full})} and advantages over centralized systems, such as scalability and robustness to faults. In a network comprised of multiple agents (e.g., data centers, sensors, vehicles, smart phones, or various IoT devices) engaged in data collection, it is sometimes impractical to collect all the information in one place. Consequently, distributed optimization techniques are currently being explored for potential use in a variety of estimation and learning problems over networks. 

This paper considers the separable optimization problem 
\begin{equation} \label{mainmodel} \min_{\bz \in \R^d} F(\bz) := \sum_{i=1}^{n}f_i(\bz), \end{equation}  where the function $f_i: \R^d \rightarrow \R$ is held only by agent $i$ in the network. We assume the agents communicate through a directed communication network, with each agent able to send messages to its out-neighbors. The agents seek to collaboratively agree on a minimizer to the global function $F(\bz)$. 

This fairly simple problem formulation is capable of capturing a variety of scenarios in estimation and learning. Informally, $\bz$ is often taken to parameterize a model, and $f_i(\bz)$ is a loss function measuring how well $\bz$ matches the data held by agent $i$. Agreeing on a minimizer of $F(\bz)$ means agreeing on a model that best explains all the data throughout the network -- and the challenge is to do this in a distributed manner, avoiding techniques such as flooding which requires every node to learn and store all the data throughout the network. For more details, we refer the reader to the recent survey by \cite{nedic2018network}. 

In this work, we will consider a fairly harsh network environment, including message losses, delays, asynchronous updates, and directed communication. The function $F(\bz)$ will be assumed to be strongly convex with the individual functions $f_i(\bz)$ having a Lipschitz continuous gradient. We will also assume that, at every time step, node $i$ can obtain a noisy gradient of its function $f_i(\bz)$. Our goal will be to investigate to what extent distributed methods can  remain competitive with their centralized counterparts in spite of these obstacles. 

\subsection{Literature review} Research on models of distributed optimization dates back to the 1980s, see \cite{tsitsiklis1986distributed}. The separable model of \eqref{mainmodel} was first formally analyzed in \cite{nedic2009distributed}, where performance guarantees on a fixed-stepsize subgradient method were obtained. The literature on the subject has exploded since, and we review here only the papers closely related to our work. We begin by discussing works that have focused on the effect of harsh network conditions. 

A number of recent papers have studied asynchronicity in the context of distributed optimization. It has been noted that asynchronous algorithms are often preferred to synchronous ones, due to the difficulty of perfectly coordinating all the agents in the network, e.g., due to clock drift.
\as{Papers by \cite{recht2011hogwild, li2014scaling, agarwal2011distributed, lian2015asynchronous} and \cite{feyzmahdavian2016asynchronous} study asynchronous parallel optimization methods in which different processors have access to a shared memory or parameter server.
\cite{recht2011hogwild} present a scheme called HOGWILD!, in which processors have access to the same shared memory with the possibility of overwriting each other's work.
\cite{li2014scaling} proposes a parameter server framework for distributed machine learning.
\cite{agarwal2011distributed} analyze the convergence of gradient-based optimization algorithms whose updates depend on delayed stochastic gradient information due to asynchrony.
\cite{lian2015asynchronous} improve on the earlier work by \cite{agarwal2011distributed}, and study two asynchronous parallel implementations of Stochastic Gradient (SG) for nonconvex optimization; establishing an $\O_k(1/\sqrt{k}$) convergence rate for both algorithms.
\cite{feyzmahdavian2016asynchronous} propose an asynchronous
mini-batch algorithm that eliminates idle waiting and allows workers to run at their maximal update rates.}

\as{The works mentioned above consider a \textit{centralized} network topology, i.e., there is a central node (parameter server or shared memory) connected to all the other nodes. On the other hand, in a \textit{decentralized} setting, nodes communicate with each other over a connected network without depending on a central node (see Figure \ref{fig: cen vs dec}). This setting reduces the communication load on the central node, is not vulnerable to failures of that node, and is more easily scalable.}

\begin{figure}
	\centering
	\begin{subfigure}[t]{0.40\textwidth}
		\includegraphics[width=\textwidth]{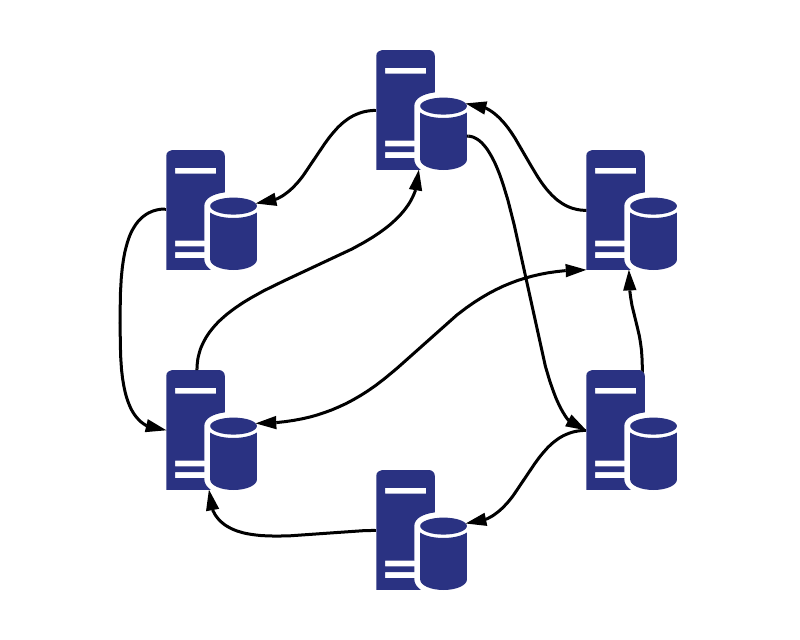}
		\caption{Decentralized network topology.}
	\end{subfigure}
	\hspace{0.05\textwidth}
	\begin{subfigure}[t]{0.40\textwidth}
		\includegraphics[width=\textwidth]{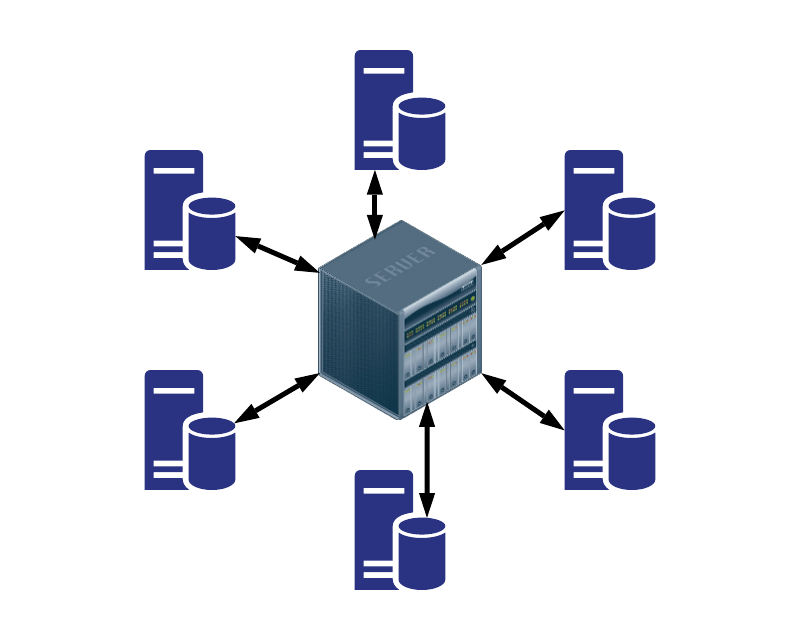}
		\caption{Centralized network topology.}
	\end{subfigure}
	\caption{Different network topologies.}
	\label{fig: cen vs dec}
\end{figure}

For analysis of how \as{decentralized} asynchronous methods perform we refer the reader to \cite{mansoori2017superlinearly, tsitsiklis1986distributed, srivastava2011distributed, assran2018asynchronous, nedic2011asynchronous, wu2018decentralized} and \cite{tian2018asy}. We note that of these works only \cite{tian2018asy} is able to obtain an algorithm which agrees on a global minimizer of \eqref{mainmodel} with non-random asynchronicity, under the assumptions of strong convexity, noiseless gradients and possible delays. On the other hand, the papers \cite{nedic2011asynchronous} and \cite{wu2018decentralized} obtain convergence in this situation under assumptions of natural randomness in the algorithm: the former assumes randomly failing links while the latter assumes that nodes make updates in random order.

The study of distributed separable optimization over directed graphs was initiated in \cite{tsianos2012push}, where a distributed approach based on dual averaging with convex functions over a fixed graph was proposed and shown to converge at an $\O_k(1/\sqrt{k})$ rate. Some numerical results for such methods were reported in \cite{tsianos2012consensus}. In \cite{nedic2015distributed}, a method based on plain gradient descent converging at a rate of $\O_k((\ln k)/\sqrt{k})$ was proposed over time-varying graphs. This was improved in \cite{nedic2016stochastic} to $\O_k((\ln k)/k)$ for strongly convex functions with noisy gradient samples.
More recent works on optimization over directed graphs are \cite{akbari2017distributed}, which considered online convex optimization in this setting, and \cite{assran2018asynchronous}, which considered combining directed graphs with delays and asynchronicity. The main tool for distributed optimization is the so-called ``push sum'' method introduced in  \cite{kempe2003gossip}, which is widely used to design communication and optimization schemes over directed graphs. More recent references are \cite{benezit2010weighted, hadjicostis2016robust}, which provide a more modern and general analysis of this method, and the most comprehensive reference on the subject is the recent monograph by \cite{hadjicostis2018distributed}. We also mention \cite{xi2017dextra, xi2018add, nedic2017achieving}, where an approach based on push-sum was explored. A parallel line of work in this setting based on the the ADMM model, where updates are allowed to include a local minimization step, was explored in \cite{brisimi2018federated, chang2016asynchronousa, chang2016asynchronousb} and \cite{hong2017distributed}.

The reason directed graphs present a problem is because much of distributed optimization relies on the primitive of ``multiplication by a doubly stochastic matrix:" given that each node of a network holds a number $x_i$, the network needs to compute $y_i$, where $\bx = (x_1,\ldots,x_n)^{\top}$, $\by = (y_1,\ldots,y_n)^{\top}$ and $\by=\bf W \bx$ for some doubly stochastic matrix $\bf W$ with positive spectral gap. This is pretty easy to accomplish over undirected graphs (see \cite{nedic2018network}) but not immediate over directed graphs. A parallel line of research focuses on distributed methods for constructing such doubly stochastic matrices over directed graphs -- we refer the reader to \cite{dominguez2013distributed, gharesifard2012distributed, dominguez2014convergence}. Unfortunately, to the authors' best knowledge, no explicit and favorable convergence time guarantees are known for this procedure.  Another line of work (\cite{xi2017distributed}) takes a similar approach, based on construction of a doubly stochastic matrix with positive spectral gap after the introduction of auxiliary states. 
\as{Among works with undirected graphs, \cite{seaman2017optimal} derived the optimal convergence rates for smooth and strongly convex functions and introduced the multi-step dual accelerated (MSDA) algorithm with optimal linear convergence rate in the deterministic case.}

Dealing with message losses has always been a challenging problem for multi-agent optimization protocols. Recently, \cite{hadjicostis2016robust} resolved this issue rather elegantly for the problem of distributed average computation  by having nodes exchange certain running sums. It was shown in \cite{hadjicostis2016robust} that the introduction of these running sums is equivalent to a lossless algorithm on a slightly modified graph. We also refer the reader to the follow-up papers by \cite{su2016non, su2016fault, su2017reaching}. We will use the same approach in this work to deal with message losses. 

In many applications, calculating the exact gradients can be computationally very expensive or impossible \cite{lan2017communication}. In one possible scenario, nodes are sensors that collect measurements at every step, which naturally corrupts all the data with noise. Alternatively, communication between agents may insert noise into information transmitted between them. Finally,  when $f_i(\bz)$ measures the fit of a model parameterized by the vector $\bz$ to the data of agent $i$, it may be  efficient for agent $i$ to randomly select a subset of its data and compute an estimate of the gradient based on only those data points (\cite{alpcan2009distributed}). Motivated by these considerations, a literature has arisen studying the effects of stochasticity in the gradients. For example, \cite{srivastava2011distributed} showed convergence of an asynchronous algorithm for constrained distributed stochastic optimization, under the presence of local noisy communication in a random communication network. In 
\cite{pu2018distributed},  two distributed stochastic gradient  methods were introduced, and their convergence to a neighborhood of the global minimum (under constant step-size) and to the global minimum (under diminishing stepsize) was analyzed.
\as{In work by \cite{sirb2016consensus}, convergence of asynchronous decentralized optimization using delayed stochastic gradients has been shown.}

The algorithms we will study here for stochastic gradient descent are based on the standard ``consensus+gradient descent" framework: nodes will take steps in the direction of their gradients and then ``reconcile'' these steps by moving in the directions of an average of their neighbors in the graph. We refer the reader to \cite{nedic2018network,yuan2016convergence}, for a more recent and simplified analysis of such methods. It is also possible to take a more modern approach, pioneered in \cite{shi2015extra}, of using the past history to make updates; such schemes have been shown to achieve  superior performance in recent years (see \cite{shi2015extra, sun2016distributed, oreshkin2010optimization, nedic2017achieving, xi2017dextra, xi2018add, qu2017harnessing, xu2015augmented,qu2017accelerated, di2016next}); we refer the reader to \cite{pu2018distributed} which took this approach. 

One of our main concerns in this paper is to develop decentralized optimization methods which perform as well as their centralized counterparts. Specifically, we will compare the performance of a distributed method for \eqref{mainmodel} on a network of $n$ nodes with the performance of a  centralized method which, at every step, can query all $n$ gradients of the functions $f_1(\bz), \ldots, f_n(\bz)$. Since the distributed algorithm gets noise-corrupted gradients, so should the centralized method. Thus, the natural approach is to compare the distributed method to centralized gradient descent which moves in the direction of the sum of the gradients of $f_1(\bz), \ldots, f_n(\bz)$. This method of comparison keeps the ``computational power'' of the two nodes identical. 

Traditionally, the bounds derived on distributed methods were considerably worse than those derived for centralized methods. For example, the papers by \cite{nedic2015distributed, nedic2016stochastic} had bounds for distributed optimization over directed graphs that were worse than the comparable centralized method (in terms of rate of error decay) by a multiplicative factor that, in the worst case, could be as large as $n^{\O(n)}$. This is typical over directed graphs, though better results are possible over undirected graphs. For example, in \cite{olshevsky2017linear}, in the model of noiseless, undelayed, synchronous communication over an undirected graph, a distributed subgradient method was proposed whose performance, relative to a centralized method with the same computational power, was worse by a multiplicative factor of $\sqrt{n}$. 

The breakthrough papers by \as{\cite{chen2015learning};} \cite{pu2017flocking, morral2017success}, were the first to address this gap. \as{These} papers studied the model where gradients are corrupted by noise, which we also consider in this paper. \as{\cite{chen2015learning} examined the mean-squared stability and convergence of distributed strategies with fixed step-size over graphs and showed the same performance level as that of a centralized strategy, in the small step-size regime.} In \cite{pu2017flocking} it was shown that, for a certain stochastic differential equation paralleling network gradient descent, the performance of centralized and distributed methods were comparable. In \cite{morral2017success}, it was proved, for the first time, that  distributed gradient descent with an appropriately chosen step-size, asymptotically performs similarly to a centralized method that takes steps in the direction of the sum of the noisy gradients, \as{assuming iterates will remain bounded almost surely}. This was the first analysis of a decentralized method for computing the \as{\textit{optimal}} solution with performance bounds matching its centralized counterpart.

Both \cite{pu2017flocking} and \cite{morral2017success}  were over fixed, undirected graphs with no message loss or delays or asynchronicity. As shown in the  paper by \cite{morral2012distributed}, this turns out to be a natural consequence of the analysis of those methods. Indeed, on a technical level, the advantage of working over undirected graphs is that they allow for easy distributed multiplication by doubly-stochastic matrices; it was shown in \cite{morral2012distributed} that if this property holds only in expectation --  that is, if the network nodes can multiply by random stochastic matrices that are only doubly stochastic in expectation --  distributed gradient descent will not perform comparably to its centralized counterpart. 

\as{In parallel to this work, and in order to reduce communication bottlenecks, \cite{koloskova2019decentralized} propose a Decentralized SGD with communication compression that can achieve the centralized baseline convergence rate, up to a constant factor.}
\as{When the objective functions are smooth but not necessarily convex, \cite{lian2017can} show that Decentralized Parallel Stochastic Gradient Descent (D-PSGD)  can asymptotically perform comparably to Centralized PSGD in total computational complexity. However, they argue that D-PSGD requires much less communication cost on the busiest node and hence, can outperform C-PSGD in certain communication regimes.}
\as{Again, both \cite{koloskova2019decentralized} and \cite{lian2017can} are over fixed undirected graphs, without delays, link failures or asynchronicity. The follow-up work by \cite{lian2017asynchronous}, extends the D-PSGD to the asynchronous case.}

\subsection{Our Contribution} 

We propose an algorithm which we call {\em Robust Asynchronous Stochastic Gradient Push (RASGP)} for distributed optimization from noisy gradient samples {\em over directed graphs with message losses, delays, and asynchronous updates}. We will assume gradients are corrupted with additive noise represented by independent random variables, with bounded support, and with finite variance at node $i$ denoted by $\sigma_i^2$. Our main result is that the RASGP performs as well as the best bounds on centralized gradient descent that moves in the direction of the sum of noisy gradients of $f_1(\bz), \ldots, f_n(\bz)$. Our results also hold if the underlying graphs are time-varying as long as there are no message losses. We give a  brief technical overview of this result next. 

We will assume that each function $f_i(\bz)$ is $\mu_i$-strongly convex with $L_i$-Lipschitz gradient, where $\sum_i \mu_i > 0$ and $L_i > 0$, $i =1, \ldots, n$.  The RASGP will have every node maintain an estimate of the optimal solution which will be updated from iteration to iteration; we will use $\bz_i(k)$ to denote the value of this estimate held by node $i$ at iteration $k$. We will show that, for each node $i=1, \ldots, n$, 
\begin{equation} \label{shortmainresult} \E \left[\Vert\bz_i(k) - \bz^*\Vert_2^2 \right] = \frac{L_u \sum_{i=1}^n \sigma_i^2}{k (\sum_{i=1}^n \mu_i)^2} + \O_k \left( \frac{1}{k^{1.5}} \right), \end{equation}
where $\bz^*:= \arg\min F(\bz)$ and $L_u$ is the \emph{degree of asynchronicity}, defined as the maximum number of iterations between two consecutive updates of any agent.
The leading term matches the best bounds for (centralized)  gradient descent that takes steps in the direction of the sum of the noisy gradients of $f_1(\bz), \ldots, f_n(\bz)$, every $k/L_u$ iterations  (see \cite{nemirovski2009robust, rakhlin2012making}). Asymptotically, the performance of the RASGP is network independent: indeed, the only effect of the network or the number of nodes is on the constant factor within the $\O_k\left(1/k^{1.5}\right)$ term above. The asymptotic scaling as $\O_k(1/k)$ is optimal in this setting (\cite{rakhlin2012making}). 

Consider the case when all the functions are identical, i.e., $f_1(\bz) = \cdots = f_n(\bz)$, and $L_u = 1$. In this case, letting $\mu = \mu_i$ and $\sigma = \sigma_i$, we have that for each $i=1, \ldots, n$,   \eqref{shortmainresult} reduces to 
\[ \E \left[\Vert\bz_i(k) - \bz^*\Vert_2^2 \right] = \frac{\sigma^2/n}{ k \mu^2} + \O_k \left( \frac{1}{k^{1.5}} \right). \] In other words, asymptotically we get the  variance reduction of  a centralized method that simply averages the $n$ noisy gradients at each step. 

The implication of this result is that one can get the benefit of having $n$ independent processors computing noisy gradients in spite of all the usual problems associated with communications over a network (i.e., message losses, latency, asynchronous updates, one-way communication). Of course, the caveat is that one must wait sufficiently long for the asymptotic decay to ``kick in,'' i.e., for the second term on the right-hand side of \eqref{shortmainresult} to become negligible compared to the first. We leave the analysis of the size of this transient period to future work and note here that it {\em will} depend on the network and the number of nodes.\footnote{It goes without saying that no analysis of distributed optimization can be wholly  independent of the network or the number of nodes. Indeed, in a network of $n$ nodes, the diameter can be as large as $n-1$, which means that, in the worst case, no bounds on global performance can be obtained during the first $n-1$ steps of any algorithm.} 

The RASGP is a variation on the usual distributed gradient descent where nodes mix consensus steps with steps in the direction of their own gradient, combined with a new step-size trick to deal with asynchrony. It is presented as Algorithm \ref{alg: optimization} in Section \ref{sec: optimization}. For a formal statement of the results presented above, we refer the reader to Theorem \ref{The: Opt conv} in the body of the paper.

We briefly mention two caveats. The first is that implementation of the RASGP requires each node to use the quantity $\sum_{i=1}^n \mu_i/n$ in  setting its local stepsize. This is not a problem in the setting when all functions are the same, but, otherwise, $\sum_{i=1}^n \mu_i/n$ is a global quantity not immediately available to each node. Assuming that node $i$ knows $\mu_i$, one possibility is to use average consensus to compute this quantity in a distributed manner before running the RASGP (for example using the algorithm described in  Section \ref{sec: push-sum} of this paper). The second caveat is that, like all algorithms based on the push-sum method, the RASGP requires each node to know its out-degree in the communication graph. 

\subsection{Organization of this paper} 

We conclude this Introduction with  Section \ref{sec: problem formulation}, which describes the basic notation we will use throughout the remainder of the paper. Section \ref{sec: push-sum} does not deal directly with the distributed optimization problem we have discussed, but rather introduces the problem of computing the average in the fairly harsh network setting we will consider in this paper. This is an intermediate problem we need to analyze on the way to our main result. Section \ref{sec: optimization} provides the RASGP algorithm for distributed optimization, and then states and proves our main result, namely the asymptotically network-independent and optimal convergence rate. Results from numerical simulations of our algorithm to illustrate its performance are provided in Section \ref{sec: numerical}, followed by conclusions in Section \ref{sec: conclusion}.

\subsection{Notations and definitions} \label{sec: problem formulation}
We assume there are $n$ agents $\V=\{1, \ldots, n\}$, communicating through a fixed directed graph $\G=(\V,\Ed)$, where $\Ed$ is the set of directed arcs. We assume $\G$ does not have self-loops and is strongly connected.

For a matrix $\mathbf{A}$, we will use $A_{ij}$ to denote its $(i,j)$th entry. Similarly, $v_i$ and $[v]_i$ will denote the $i$th entry of a vector $\bv$.
A matrix is called \textit{stochastic} if it is non-negative and the sum of the elements of each row equals to one. A matrix is \textit{column stochastic} if its transpose is stochastic.
To a non-negative matrix $\mathbf{A} \in \R^{n\times n}$ we associate a directed graph $\mathcal{G}_\mathbf{A}$ with vertex set $\mathcal{V}_{\mathbf{A}}=\{1,2,\ldots,n\}$ and edge set $\mathcal{E}_\mathbf{A}=\{(i,j)\vert  A_{ji}>0 \}$.  In general, such a graph might contain self-loops. Intuitively, this graph corresponds to the information flow in the update  $\bx(k+1) = \mathbf{A} \bx(k)$; indeed, $(i,j) \in \mathcal{E}_\mathbf{A}$ if the $j$th coordinate of $\bx(k+1)$ depends on the $i$th coordinate of $\bx(k)$ in this update. 

Given a sequence of matrices $\mathbf{A}(0),\mathbf{A}(1),\mathbf{A}(2),\ldots$, we denote by $\mathbf{A}^{k_2:k_1},k_2\geq k_1$, the product of elements $k_1$ to $k_2$ of the sequence, inclusive, in the following order:
\begin{displaymath}
\mathbf{A}^{k_2:k_1}=\mathbf{A}(k_2)\mathbf{A}(k_2-1)\cdots \mathbf{A}(k_1).
\end{displaymath}
Moreover, $\mathbf{A}^{k:k}=\mathbf{A}(k)$.

Node $i$ is an \textit{in-neighbor} of node $j$, if there is a directed link from $i$ to $j$. Hence, $j$ would be an \textit{out-neighbor} of node $i$. We denote the set of in-neighbors and out-neighbors of node $i$ by $N_i^{-}$ and $N_i^{+}$, respectively. Moreover, we denote the number of in-neighbors and out-neighbors of node $i$ with $d_i^{-}$ and $d_i^{+}$, as its \textit{in-degree} and \textit{out-degree}, respectively.

By $x_{\min}$ and $x_{\max}$ we denote $\min_i x_i$ and $\max_i x_i$ respectively, over all possible indices unless mentioned otherwise. We denote a $n \times 1$ column vector of all ones or zeros by $\mathbf{1}_n$ and $\mathbf{0}_n$, respectively. We will remove the subscript when the size is clear from the context. 

Let $\bv \in \R^d$ be a vector. We denote by $\bv^-\in \R^d$ a vector of  the same length such that
\begin{align*}
	v_i^-=\begin{cases}
		1/v_i, \qquad &\text{if } v_i\neq 0,\\
		0,\qquad &\text{if } v_i= 0.
	\end{cases}
\end{align*}

For all the algorithms we describe, we sometimes use the notion of \textit{mass} to denote the value an agent holds, sends or receives. With that in mind, we can think of a value being sent from one node, as a mass being transferred. 

We use $\Vert.\Vert_p$ to denote the $l_p$-norm of a vector. We sometimes drop the subscript when referring to the Euclidean $l_2$ norm.

\section{Push-sum with delays and link failures}\label{sec: push-sum}
In this section we introduce the Robust Asynchronous Push-Sum algorithm (RAPS) for distributed average computation and prove its exponential convergence.  Convergence results proved for this algorithm will be used later when we turn to distributed optimization. The algorithm relies heavily on ideas from \cite{hadjicostis2016robust} to deal with message losses, delays, and asynchrony. The conference version of this paper  \cite{olshevsky2018fully} developed RAPS for the delay-free case, and this section may be viewed as an extension of that work.  

Pseudocode for the algorithm is given in the box for Algorithm \ref{alg: raps}.  We begin by outlining the operation of the algorithm. Our goal in this section is to compute the average of vectors, one held by each node in the network, in a distributed manner. However, since the RAPS algorithm acts separately in each component, we may, without loss of generality, assume that we want to average scalars rather than vectors. The scalar held by node $i$ will be denoted by $x_i(0)$.

\as{Without loss of generality, we define an iteration by descretizing time into time slots indexed by $k=0,1,2,\ldots$. We assume that during each time slot every agent makes at most one update and processes messages sent in previous time slots.}

In the setting of no message losses, no delays, no asynchrony, and a fixed, regular, undirected communication graph, the RAPS can be shown to be equivalent to the much simpler iteration 
\[ \bx(t+1) = {\bf W} \bx(t), \] where $\bf W$ is an irreducible, doubly stochastic matrix with positive diagonal; standard Markov chain theory implies that $x_i(t) \rightarrow (1/n)\sum_{i=1}^n x_i(t)$ in this setting. RAPS does essentially the same linear update, but with a considerable amount of modifications. In particular, we use  the central idea of the classic push-sum method (\cite{kempe2003gossip}) to deal with directed communication, which suggests to have a separate update equation for the $y$-variables, which informs us how we should rescale the $x$-variables; as well as the central idea of \cite{hadjicostis2018distributed}, which is to repeatedly broadcast sums of previous messages to provide robustness against message loss. \as{While the algorithm in \cite{hadjicostis2018distributed} handles message losses in a synchronous setting, RAPS can handle delays as well as asynchronicity.}

\as{Before getting into details, let us provide a simple intuition behind the RAPS algorithm. Each agent $i$ holds a value (mass) $x_i$ and $y_i$. At the beginning of every iteration, $i$ wants to split its mass between itself and its out-neighbors $j \in N_i^+$. However, to handle message losses, it sends the accumulated $x$ and $y$ mass (running sums which we denote by  $\phi_i^x$ and $\phi_i^y$), that $i$ wants to transfer to each of its neighbors, from the start of the algorithm.
Therefore, when a neighbor $j$ receives a new accumulated mass from $i$, it stores it at $\rho_{ji}^*$ and by subtracting the previous accumulated mass $\rho_{ji}$ it had received from $i$, $j$ obtains all the mass that $i$ has been trying to send since its last successful communication. Then, $j$ updates its $x$ and $y$ mass by adding the new received masses, and finally, updates its estimate of the average to $x/y$.
To handle delays and asynchronicity, timestamps $\kappa_i$ are attached to messages outgoing from $i$.}

\begin{algorithm}[h]
	\caption{Robust Asynchronous Push-Sum (RAPS)}
	\begin{algorithmic}[1]\label{alg: raps}
		\STATE Initialize the algorithm with $\by(0)=\mathbf{1}$, $\phi_i(0)=0$,  $\forall i\in\{1,\ldots,n\}$ and $\rho_{ij}(0)=0$, $\kappa_{ij}(0)=0$, $\forall (j,i)\in \Ed$.
		\STATE At every iteration $k=0,1,2,\ldots$, for every node $i$:
		\IF{node $i$ wakes up}\label{line3_alg1}
		\STATE $\kappa_i \leftarrow k$;\label{line4_alg1}
		\STATE $\phi_i^x \leftarrow \phi_i^x +\frac{x_i}{d_i^{+}+1}$,
		$\phi_i^y \leftarrow \phi_i^y +\frac{y_i}{d_i^{+}+1}$;
		\STATE $x_i \leftarrow \frac{x_i}{d_i^{+}+1}$,
		$y_i \leftarrow \frac{y_i}{d_i^{+}+1}$;
		\STATE Node $i$ broadcasts $(\phi_i^x,\phi_i^y,\kappa_i)$ to its out-neighbors in $N_i^{+}$.
		\STATE \textbf{Processing the received messages}
		\FOR{$(\phi_j^x,\phi_j^y,\kappa_j')$ in the inbox}
		\IF{$\kappa_j'>\kappa_{ij}$} \label{discard}
		\STATE $\rho_{ij}^{*x} \leftarrow \phi_j^x$,
		$\rho_{ij}^{*y} \leftarrow \phi_j^y$;
		\STATE $\kappa_{ij} \leftarrow \kappa_j'$;
		\ENDIF
		\ENDFOR
		\STATE $x_i \leftarrow x_i + \sum_{j\in N_i^{-}}\left(\rho_{ij}^{*x}-\rho_{ij}^x \right)$,
		$y_i \leftarrow y_i + \sum_{j\in N_i^{-}}\left(\rho_{ij}^{*y}-\rho_{ij}^y \right)$;
		\STATE $\rho_{ij}^x \leftarrow \rho_{ij}^{*x}$,
		$\rho_{ij}^y \leftarrow \rho_{ij}^{*y}$,\label{rhoup}
		\STATE $z_i \leftarrow \frac{x_i}{y_i}$; \label{line17_alg1}
		\ENDIF
		\STATE Other variables remain unchanged.
	\end{algorithmic}
\end{algorithm}

The pseudocode for the algorithm may appear complicated at first glance; this is because of the considerable complexity required to deal with directed communications, message losses, delays, and asynchrony.

We next describe the algorithm in \as{more detail}. First, in the course of executing the algorithm, every agent $i$ maintains scalar variables $x_i$, $y_i$, $z_i$, $\phi^x_i$, $\phi^y_i$, $\kappa_i$, $\rho_{ij}^x$, $\rho_{ij}^y$ and $\kappa_{ij}$ for $(j,i)\in \Ed$. The variables $x_i$ and $y_i$ have the same evolution, \as{however $y_i$ is initialized as $1$}. Therefore, to save space in describing and analyzing the algorithm, we will use the symbol $\theta$, when a statement holds for both $x$ and $y$. Similarly, when a statement is the same for both variables $x$ and $y$, we will remove the superscripts $x$ or $y$.
For example, the initialization $\rho_{ji}(0)=0$ in the beginning of the algorithm means both $\rho_{ji}^x(0)=0$ and $\rho_{ji}^y(0)=0$.

We briefly mention the intuitive meaning of the various variables. The number $z_i$ represents node $i$'s estimate of the initial average. The counter $\phi^{\theta}_i(k)$ is the total $\theta$-value sent by $i$ to each of its neighbors from time $0$ to $k-1$. Similarly, $\rho^{\theta}_{ij}(k)$ is the total $\theta$ value that $i$ has received from $j$ up to time $k-1$. 
The integer $\kappa_i$ is a timestamp that $i$ attaches to its messages, and the number $\kappa_{ij}$ tracks the latest timestamp $i$ has received from $j$.

To obtain an intuition for how the algorithm uses the counters $\phi^{\theta}_i(k)$ and $\rho^{\theta}_{ij}(k)$, note that, in line 15 of the algorithm,  node $i$ effectively figures out the last $\theta$ value sent to it by each of its in-neighbors $j$, by looking at the increment to the $\rho^{\theta}_{ij}$.  This might seem needlessly involved, but,  the underlying reason is that this approach introduces robustness to message losses.

We next describe in words what the pseudocode above does.  At every iteration $k$, if agent $i$ wakes up, it performs the following actions. First, it divides its values $x_i, y_i$ into $d_i^+ + 1$ parts and broadcasts these to its out-neighbors; actually, what it broadcasts are the accumulated running sums $\phi_i^x$ and $\phi_i^y$. Following \cite{kempe2003gossip}, this is sometimes called the ``push step.'' 

Then, node $i$ moves on to process  the messages in its inbox in the following way. If agent $i$ has received a message from node $j$ that is newer than the last one it has received before, it will store that message in $\rho_{ij}^*$ and discard the older messages. Next, $i$ updates its $x$ and $y$ variables by adding the difference of $\rho_{ij}^*$ with the older value $\rho_{ij}$, for all in-neighbors $j$. As mentioned above, this difference is equal to the new mass received. Next, $\rho_{ij}^*$ overwrites $\rho_{ij}$ in the penultimate step. The last step of the algorithm sets $z_i$ to be the rescaled version of $x_i$: $z_i = x_i/y_i$. 

In the remainder of this section, we provide an analysis of the RAPS algorithm, ultimately showing that it converges geometrically to the average in the presence of message losses, asynchronous updates, delays, and directed communication. Our first step is to formulate the RAPS algorithm in terms of a linear update (i.e., a matrix multiplication), which we do in the next subsection.

\subsection{Linear formulation}
Next we show that, after introducing some new auxiliary variables,  Algorithm \ref{alg: raps} can be written in terms of a classical push-sum algorithm (\cite{kempe2003gossip}) on an augmented graph. Since the $y$-variables have the same evolution as the $x$-variables, here we only analyze the $x$-variables.

In our analysis, we will associate with each message an {\em effective delay}. 
If a message is sent at time $k_1$ and is ready to be processed at time $k_2$, then $k_2-k_1 \geq 1$ is the effective delay experienced by that message. Those messages that are discarded  will not have an effective delay associated with them and are considered as lost.

Next, we will state our assumptions on connectivity, asynchronicity, and message loss. 

\begin{assumption}\label{asm: connvectivity}
	Suppose:
	\begin{enumerate}
		\item[(a)]
		Graph $\G$ is strongly connected and does not have self-loops.
		\item[(b)]
		The  delays on each link are bounded above by some $L_{\rm del} \geq 1$.
		\item[(c)]
		Every agent wakes up and performs updates at least once every $L_u\geq 1$ iterations.
		\item[(d)]
		Each link fails at most $L_f\geq 0$ consecutive times.
		\item[(e)] \label{ass:e}
		Messages arrive in the order of time they were sent. In other words, if messages are sent from node $i$ to $j$ at times $k_1$ and $k_2$ with (effective) delays $d_1$ and $d_2$, respectively, and $k_1<k_2$, then we have $k_1+d_1<k_2+d_2.$
	\end{enumerate}
\end{assumption}

One consequence of Assumption \ref{asm: connvectivity} is that the effective delays associated with each message that gets through are bounded above by $L_d := L_{\rm del} + L_u - 1$. Another consequence is that for each $(i,j)\in \Ed$, $j$ receives a message from $i$ successfully, at least once every $L_s$ iterations where \begin{equation} \label{eq:Ls} L_s:=L_u(L_f+1)+L_d\geq 2. \end{equation} 

Part (e) of Assumption \ref{ass:e} can be assumed without loss of generality. Indeed, observe that outdated messages automatically get discarded in Line \ref{discard} of our algorithm. For simplicity, it is convenient to think of those messages as lost. Thus, if this assumption fails in practice, the algorithm will perform exactly as if it had actually held in practice due to Line \ref{discard}. Making this an assumption, rather than a proposition, lets us slightly simplify some of the arguments and avoid some redundancy throughout this paper.

Let us introduce the following indicator variables:
$\tau_i(k)$ for $i\in\{1,\ldots,n\}$ which equals to $1$ if node $i$ wakes up at time $k$, and equals $0$ otherwise.
Similarly, $\tau_{ij}^l(k)$ for $(i,j)\in \Ed$, $1 \leq l\leq L_d$, which is $1$ if $\tau_i(k)=1$ {\bf and} the message sent from node $i$ to $j$ at time $k$ will arrive after experiencing an effective delay of $l$.~\footnote{Note the difference between indexing in $\tau_{ij}^l$ and $\rho_{ji}^x$, which are both defined for link $(i,j)\in \Ed$.} Note that if node $i$ wakes up at time $k$ but the \as{message it sends to $j$ is} lost, then $\tau_{ij}^l(k)$ will be zero for all $l$.

We can rewrite the RAPS algorithm  with the help of these indicator variables. Let us adopt the notation that $x_i(k)$ refers to $x_i$ at the {\bf beginning} of round $k$ of the algorithm (i.e., before node $i$ has a chance to go through the list of steps outlined in the algorithm box). We will use the same convention with all of the other variables, e.g., $y_i(k), z_i(k)$, etc. If node $i$ does not wake up at round $k$, then of course $x_i(k+1)=x_i(k)$.

Now observe that we can write \begin{equation} \phi_i^x(k+1) -\phi_i^x(k)=\tau_i(k)\frac{x_i(k)}{d_i^{+}+1}\label{eq: sigma alg}. \end{equation}
Likewise, we have
	\begin{align}
	x_i(k+1)&=x_i(k) \left( 1-\tau_i(k)+\frac{\tau_i(k)}{d_i^{+}+1} \right) + \sum_{j\in N_i^-}\left(\rho_{ij}^x(k+1)-\rho_{ij}^x(k)\right) \label{eq: x alg}, 
	\end{align} 
	which can be shown by considering each case ($\tau_i(k)=1$ or $0$); note that we have used the fact that, in the event that node $i$ wakes up at time $k$, the variable $\rho_{ij}^x(k+1)$ equals the variable $\rho_{ij}^{*x}$ during execution of Line \ref{rhoup} of the algorithm at time $k$.

Finally, we have that $\forall (i,j) \in \Ed$, the flows $\rho_{ji}^x$ are updated as follows:
	\begin{align}\label{eq: rho alg}
	\rho_{ji}^x(k+1) = \rho_{ji}^x(k) + \sum_{l=1}^{L_d} \tau_{ij}^l(k-l) \left( \phi_i^x(k+1-l) - \rho_{ji}^x(k) \right),
	\end{align}
	where we make use of the fact that the sum contains only a single nonzero term, since the messages arrive monotonically. To parse the indices in this equation, note  that  node $i$ actually broadcasts $\phi_i^x(k+1-l)$ in our notation at iteration $k-l$; by our definitions, $\phi_i^x(k-l)$ is the value of $\phi_i^x$ at the {\bf beginning} of that iteration. 
	To simplify these relations, we introduce the auxiliary variables $u_{ij}^x$ for all $(i,j)\in \Ed$, defined through the following recurrence relation: 
	\begin{align}\label{eq: U def}
	u_{ij}^x(k+1) := \left(1-\sum_{l=1}^{L_d} \tau_{ij}^l(k)\right)\Big(u_{ij}^x(k) + \phi_i^x(k+1)-\phi_i^x(k)\Big),
	\end{align}
	and initialized as $u_{ij}^x(0):=0$. Intuitively, the variables $u_{ij}^x$ represent the ``excess mass'' of $x_i$ that is yet to reach node $j$. Indeed, this quantity resets to zero whenever a message is sent that arrives at some point in the future, and otherwise is incremented by adding the broadcasted mass that is lost. Note that node $i$ never knows $u_{ij}^x(k)$, since it has no idea which messages are lost, and which are not; nevertheless, for purposes of analysis, nothing prevents us from considering these variables.

Let us also define the related quantity,
\begin{align*}
\mu_{ij}^x(k):=u_{ij}^x(k)+\phi_i^x (k+1)-\phi_i^x(k), \qquad \text{for } k\geq 0,
\end{align*}
and $\mu_{ij}^x(k):=0$ for $k < 0$. Intuitively, this quantity may be thought of as a forward-looking estimate of the mass that {\em will arrive} at node $j$,
if the message sent from node $i$ at time $k$ gets through;
correspondingly, it includes not only the previous unsent mass, but the extra mass that will be added at the current iteration.

The key variables for the analysis of our method are the variables we will denote by $x_{ij}^l(k)$. Intuitively, every time a message is sent, but gets lost, we imagine that it has instead arrived into a ``virtual node'' which holds that mass; once the next message gets through, we imagine that the virtual node has forwarded that mass to its intended destination. This idea originates from \cite{hadjicostis2016robust}.  Because of the delays, however, we need to introduce $L_d$ virtual nodes for each such event. If a message is sent from $i$ and arrives at $j$ with
effective delay $l$, we will instead imagine it is received by the virtual node $b_{ij}^l$, then sent to $b_{ij}^{l-1}$ at the next time step, and so forth until it reaches $b_{ij}^1$, and is then forwarded to its destination. These virtual nodes are defined formally later.

Putting that intuition aside, we formally define the variables  $x_{ij}^l(k)$ via the following set of recurrence relations:
\begin{align}
x_{ij}^{l}(k+1) &:= \tau_{ij}^{l}(k)\mu_{ij}^x(k), &l={L_d}, \label{eq:bd} \\
x_{ij}^l(k+1) &:=  \tau_{ij}^l(k)\mu_{ij}^x(k) + x_{ij}^{l+1}(k),  &1\leq l <{L_d}, \label{eq:prop}
\end{align}
and $x_{ij}^l(k):=0$ when both $k \leq 0$ and $l=1,\ldots,L_d$. To parse these equations, imagine what happens when a message is sent from $i$ to $j$ with effective delay of $L_d$ at time $k$. The content of this message becomes the value of $x_{ij}^{L_d}$ according to \eqref{eq:bd}; and, in each subsequent step, influences $x_{ij}^{L_d-1}, x_{ij}^{L_d-2}$, and so forth according to \eqref{eq:prop}.
Putting \eqref{eq:bd} and \eqref{eq:prop} together,  we obtain 
\begin{align}\label{eq:xl}
x_{ij}^l(k)=\sum_{t=1}^{L_d-l+1}\tau_{ij}^{t+l-1}(k-t)\mu_{ij}^x(k-t),
\end{align}
and particularly,
\begin{align}\label{eq:x1}
x_{ij}^1(k) = \sum_{t=1}^{L_d} \tau_{ij}^t(k-t)\mu_{ij}^x(k-t).
\end{align}
Note that, as is common in many of the equations we will write, only a single term in the sums can be nonzero (this is not obvious at this point and is a result of Lemma \ref{lem: delay}).

Before proceeding to the main result of this section, we state the following lemma, whose proof is immediate. 

\begin{lemma}\label{lem: delay}
	If $\tau_{ij}^l(k)=1$, the following statements are satisfied:
	\begin{enumerate}
		\item[(a)]
		$\tau_{ij}^{l'}(k)=0$ for $l'\neq l$.
		\item[(b)]
		If $l>0$, then $\tau_{ij}^{s}(k+t)=0$ for $t=1,\ldots,l$ and $s=0,\ldots,l-t$.
		\item[(c)]
		If $l<L_d$, then $\tau_{ij}^{s}(k-t)=0$ for $t=1,\ldots,L_d-l$ and $s=l+t,\ldots,L_d$.
	\end{enumerate}
\end{lemma}

\begin{lemma}\label{lem: xl>l=0}
	If $\tau_{ij}^l(k)=1$ then $x_{ij}^{l'}(k)=0$ for $l'>l$.
\end{lemma}
\begin{proof}
	By Lemma \ref{lem: delay}(c), $\tau_{ij}^{t+l'-1}(k-t)=0$ for $t\in \{1,\ldots,L_d-l'+1\}$. Hence, by \eqref{eq:xl} we have,
	\[ x_{ij}^{l'}(k)=\sum_{t=1}^{L_d-l'+1}\tau_{ij}^{t+l'-1}(k-t)\mu_{ij}^x(k-t)=0.
	\]
\end{proof}

The next lemma is essentially a restatement of the observation that the content of every $x_{ij}^{l'}$ eventually ``passes through'' $x_{ij}^1$.

\begin{lemma}\label{lem: x1=xl}
	If $\tau_{ij}^l(k-l)=1$, $l\geq 1$, we have,
	\[
	\sum_{l'=1}^{l}x_{ij}^{l'}(k-l)=\sum_{t=1}^{l}x_{ij}^1(k-t).
	\]
\end{lemma}
\begin{proof}
	We will show $x_{ij}^1(k-t)=x_{ij}^{l-t+1}(k-l)$ for $t=1,\ldots,l$. For $t=l$ the equality is trivial. Now suppose $t<l$.
	By Lemma \ref{lem: delay}(a) we have $\tau_{ij}^{l-t}(k-l)=0$. Moreover, by part (b) of the same lemma we have, $\tau_{ij}^{s'}(k-l+t')=0$ for $t'=1,\ldots,l-t-1$ and $s'=l-t-t'$. Hence, $x_{ij}^{l-t-t'+1}(k-l+t')=x_{ij}^{l-t-t'}(k-l+t'+1)$. Combining these equations for $t'=0,\ldots,l-t-1$, we get $x_{ij}^1(k-t)=x_{ij}^{l-t+1}(k-l)$.
\end{proof}

The following lemma is the key step of a linear formulation of RAPS.

\begin{lemma}\label{lem: rho u}
	For $k=0,1,\ldots$ and $(i,j)\in \Ed$ we have:
	\begin{gather}
	\rho_{ji}^x(k+1)-\rho_{ji}^x(k)=x_{ij}^1(k) \label{eq: rho x1},\\
	u_{ij}^x(k+1)+\rho_{ji}^x(k+1)+\sum_{l=1}^{L_d} x_{ij}^l(k+1)=\phi_i^x(k+1).\label{eq: u rho x}
	\end{gather} 
\end{lemma}

Parsing these equations, \eqref{eq: rho x1} simply states that the value of $x_{ij}^1(k)$ can be thought of as impacting $\rho_{ji}^x$ at time $k$; recall that the content of $x_{ij}^1(k)$ is a message that was sent from node $i$ to $j$ at time $k-l$ with an effective delay of $l$, for some $1\leq l \leq L_d$ (cf.\ \eqref{eq:x1}).
On the other hand, \eqref{eq: u rho x} may be thought of a ``conservation of mass'' equation. All the mass that has been sent out by node $i$ has either: (i) been lost (in which case it is in $u_{ij}^x$), (ii) affected node $j$ (in which case it is in $\rho_{ji}^x$), or (iii) is in the process of reaching node $j$ but delayed (in which case it is in some $x_{ij}^l$).

Although this lemma is arguably obvious, a formal proof is surprisingly lengthy. For this reason, we relegate it to the Appendix. 

We next write down a matrix form of our updates. As a first step, define the $(n+m')\times1$ column vector $\bchi(k):=[\bx(k)^{\top},\bx^1(k)^{\top},\ldots,\bx^{L_d}(k)^{\top},\mathbf{u}^x(k)^{\top}]^{\top}$, where $m':=(L_d+1)m$, $m:=|\Ed|$, $\bx(k)$ collects all $x_i(k)$, $\bx^l(k)$ collects all $x_{ij}^l(k)$ and, $\mathbf{u}^x(k)$ collects all $u_{ij}^x(k)$. 
Define $\bpsi(k)$ by collecting $y$-values similarly.

Now, we have all the tools to show the linear evolution of $\bchi(k)$. By Equations \eqref{eq: sigma alg}, \eqref{eq: x alg} and \eqref{eq: rho x1} we have, \begin{small}
	\begin{align}\label{eq: x-linear}
	x_j(k+1) &=x_j(k)\left(1-\tau_j(k)+\frac{\tau_j(k)}{d_j^{+}+1}\right) 
	+\sum_{i\in N_j^-}x_{ij}^1(k).
	\end{align}  \end{small}
Moreover, by the definitions of $x_{ij}$, $\mu_{ij}$ and \eqref{eq: sigma alg} it follows,
\begin{align}\label{eq: xl-linear}
\begin{split}
x_{ij}^{L_d}(k+1) &= \tau_{ij}^{L_d}(k) \left[u_{ij}^x(k)+\frac{x_i(k)}{d_i^{+}+1} \right],\\
x_{ij}^l(k+1) &=  \tau_{ij}^l(k) \left[u_{ij}^x(k)+\frac{x_i(k)}{d_i^{+}+1}\right] + x_{ij}^{l+1}(k).
\end{split}
\end{align}
Finally, by \eqref{eq: sigma alg} and \eqref{eq: U def} we obtain, \begin{small}
	\begin{align}\label{eq: u-linear}
	u_{ij}^x(k+1) = \big(1-\sum_{l=1}^{L_d} \tau_{ij}^l(k)\big)\Big(u_{ij}^x(k) +\tau_i(k)\frac{x_i(k)}{d_i^{+}+1}\Big).
	\end{align} \end{small}
Using \eqref{eq: x-linear} to \eqref{eq: u-linear} we can write the evolution of $\bchi(k)$ and $\bpsi(k)$ in the following linear form:
\begin{align}
\label{eq: XMX}
\begin{split}
\bchi(k+1)=\bM(k)\bchi(k),\\
\bpsi(k+1)=\bM(k)\bpsi(k),
\end{split}
\end{align}
where $\bM(k)\in \mathbb{R}^{(n+m')\times(n+m')}$ is an appropriately defined matrix.

We have thus completed half of our goal: we have shown how to write RAPS as a linear update. Next, we show that the corresponding matrices are column-stochastic.

\begin{lemma}\label{lem: M stochastic}
	$\bM(k)$ is column stochastic and its positive elements are at least $1/(\max_i\{d_i^+\}+1)$. Moreover, for $i=1,\ldots,n$, $M_{ii}(k)$ are positive.
\end{lemma}

This lemma can  be proved ``by inspection.'' Indeed, $\bM(k)$ is column stochastic if and only if, for every $\bchi(k)$, we have ${\bf 1}^T \bchi(k+1) = \1^T \bchi(k)$. Thus one just needs to demonstrate that no mass is ever ``lost,'' i.e., that a decrease/increase in the value of one node is always accompanied by an increase/decrease of the value of another node,  which can be done just by inspecting the equations. A formal proof is nonetheless given next.

\begin{proof}
	To show that $\bM(k)$ is column stochastic, we study how each element of $\bchi(k)$ influences $\bchi(k+1)$.
	
	For $i=1,\ldots,n$, the $i$th column of $\bM(k)$ represents how $x_i(k)$ influences $\bchi(k+1)$.\\
	We will use \eqref{eq: x-linear} to \eqref{eq: u-linear} to find these coefficients.
	
	First, $x_i(k)$ influences $x_i(k+1)$ with the coefficient $1-\tau_i(k)+\tau_i(k)/(d_i^++1)>0$. For $j \in N_i^+$, $x_i(k)$ influences
	$x_{ij}^l(k+1)$ by $\tau_{ij}^l(k)/(d_i^++1)$ and $u_{ij}^x(k+1)$ with coefficient $(\tau_{i}(k)-\sum_{l=1}^{L_d} \tau_{i}(k)\tau_{ij}^l(k))/(d_i^++1)$. Summing these coefficients up results in $1$.
	
	For $l=2,\ldots,L_d$, $(i,j)\in \Ed$, $x_{ij}^{l}(k)$ influences $x_{ij}^{l-1}(k+1)$ with coefficient $1$ and $x_{ij}^{1}(k)$ influences $x_{j}(k+1)$ with coefficient $1$.
	
	Finally, $u_{ij}^x(k)$ influences $x_{ij}^l(k+1)$ with coefficient $\tau_{ij}^l(k)$ and $u_{ij}^x(k+1)$ with $(1-\sum_{l=1}^d \tau_{ij}^l(k))$, which sum up to $1$.
	
	Note that all the coefficients above are at least $1/(\max_i\{d_i^+\}+1)$.
\end{proof}

An important result of this lemma is the sum preservation property, i.e.,
\begin{align}\label{eq: sum preservation}
\begin{split}
\sum_{i=1}^{n+m'} \chi_i(k) &= \sum_{i=1}^n x_i(0),\\
\sum_{i=1}^{n+m'} \psi_i(k) &= n.
\end{split}
\end{align}

For further analysis, we augment the graph $\G$ to $\H(k):=\G_{\bM(k)}=(\V_A,\Ed_A(k))$ by adding the following virtual nodes: $b_{ij}^l$ for $l=1,\ldots,L_d$ and $(i,j)\in \Ed$, which hold the values $x_{ij}^l$ and $y_{ij}^l$; We also add the nodes $c_{ij}$ for $(i,j)\in \Ed$ which hold the values $u_{ij}^x$ and $u_{ij}^y$.

In $\H(k)$, there is a link from $b_{ij}^l$ to $b_{ij}^{l-1}$ for $1<l \leq d$ and from $b_{ij}^1$ to $j$ as they forward their values to the next node. Moreover, if $\tau_{ij}^l(k)=1$ for some $1\leq l \leq L_d$, then there is a link from both $c_{ij}$ and $i$ to $b_{ij}^l$.

If $\tau_{ij}^l(k)=0$ for $1\leq l \leq L_d$ then $c_{ij}$ has a self loop, and if also $\tau_i(k)=1$, there's a link from $i$ to $c_{ij}$. All non-virtual agents $i\in \V$, have self-loops all the time (see Fig. \ref{fig: augmented graph}).

\begin{figure}[ht]
	\centering
	\begin{subfigure}[t]{0.32\textwidth}
		\includegraphics[width=\textwidth]{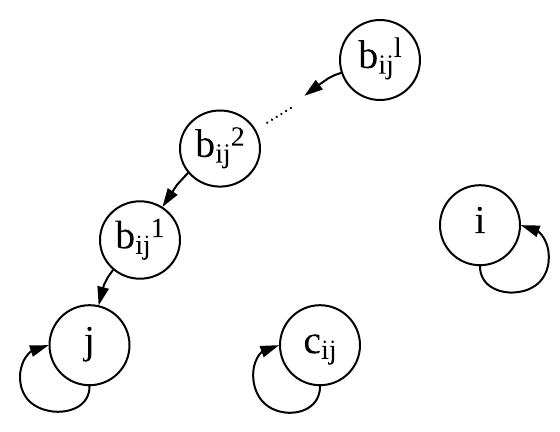}
		\caption{$\tau_i(k)=0$.}
	\end{subfigure}
	\begin{subfigure}[t]{0.32\textwidth}
		\includegraphics[width=\textwidth]{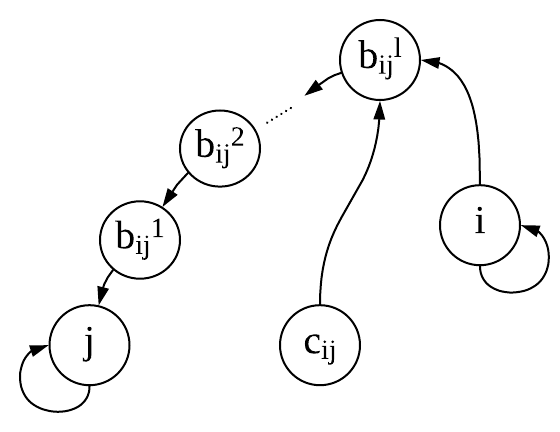}
		\caption{$\tau_{ij}^l(k)=1$.}
	\end{subfigure}
	\begin{subfigure}[t]{0.32\textwidth}
		\includegraphics[width=\textwidth]{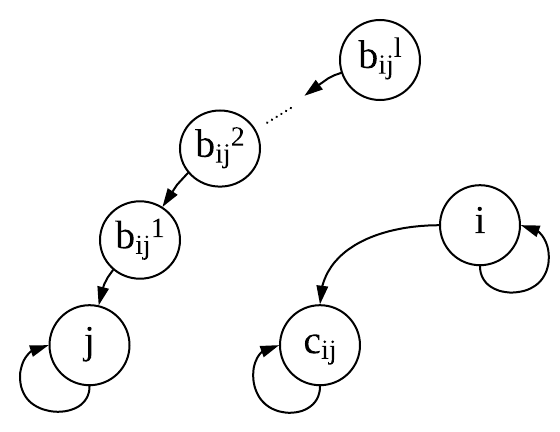}
		\caption{$\tau_i(k)=1$, $\tau_{ij}^l(k)=0, \forall l$.}
	\end{subfigure}
	\caption{Augmented graph $\H(k)$ for different scenarios.}
	\label{fig: augmented graph}
\end{figure}

Recursions \eqref{eq: XMX} and Lemma \ref{lem: M stochastic} may thus be interpreted as showing that the RAPS algorithm can be thought of as a push-sum algorithm over the augmented graph sequence $\{\H(k)\}$, where each agent (virtual and non-virtual) holds an $x$-value and a $y$-value which evolve similarly and in parallel.

\subsection{Exponential convergence}
The main result of this section is exponential convergence of RAPS to initial average, stated next.
\begin{theorem}\label{The: geometric push-sum}
	Suppose Assumption \ref{asm: connvectivity} holds. Then RAPS converges exponentially to the initial mean of agent values. i.e.,
	\begin{align*}
	\left| z_i(k)-\frac{1}{n}\sum_{i=1}^n x_i(0) \right| \leq \delta \lambda^k \Vert \bx(0) \Vert_1,
	\end{align*}
	where $\delta:=\frac{1}{1-n\alpha^6}$, $\lambda:=(1-n\alpha^6)^{1/(2nL_s)}$ and $\alpha:=(1/n)^{nL_s}$.
\end{theorem} 
\as{It is worth mentioning that even though $1/(1 - \lambda) = \O(n^{p(n)})$ where $p(n) = \O(n)$, this is a bound for a worst case scenario and on average, as it can be seen in numerical simulations, RAPS performs better. Moreover, when the graph $\G$ satisfies certain properties, such as regularity, and also there is no link delays and failures, we have $1/(1 - \lambda) = \O(n^3)$ (see Theorem 1 in \cite{nedic2016stochastic}). More broadly, that paper establishes that $1/(1-\lambda)$ will scale with the mixing rate of the underlying Markov process.}

Unfortunately, this theorem does not follow immediately from standard results on exponential convergence of push-sum. The reason is that the connectivity conditions assumed for such theorems are not satisfied here: there will not always be paths leading to virtual nodes from non-virtual nodes. Nevertheless, with some suitable modifications, the existence of paths from virtual nodes to other virtual nodes is sufficient, as we will show next.

Before proving the theorem, we need the following lemmas and definitions.
Given a sequence of graphs $\G^0,\G^1,\G^2,\ldots$, we will say node $b$ is reachable from node $a$ in time period $k_1$ to $k_2$ ($k_1<k_2$), if there exists a sequence of directed edges $e_{k_1},e_{k_1+1},\ldots,e_{k_2}$ such that $e_k$ is in $\G^k$, the destination of $e_k$ is the origin of $e_{k+1}$ for $k_1\leq k<k_2$, and the origin of $e_{k_1}$ is $a$ and the destination of $e^{k_2}$ is $b$.

Our first lemma provides a standard lower bound on the entries of the column-stochastic matrices from \eqref{eq: XMX}. 

\begin{lemma}\label{lemma:positive n-rows}
	$\bM^{k+nL_s-1:k}$ has positive first $n$ rows, for  any $k\geq 0$. The positive elements of this matrix are at least $$\alpha = (1/n)^{nL_s}.$$
\end{lemma}
\begin{proof}
	By Lemma \ref{lem: M stochastic}, each node $j\in \V$ has self-loops at every iteration in the augmented graph $\H$. Since $\G$ is strongly connected, the set of reachable non-virtual nodes from any node $a_h \in \V_A$ strictly increases every $L_s$ iterations.
	Hence, $\bM^{k+nL_s-1:k}$ has positive first $n$ rows.
	Moreover, since all positive elements of $M$ are at least $1/n$, the positive elements of $\bM^{k+nL_s-1:k}$ are at least $(1/n)^{nL_s}$. 
\end{proof}
Next, we give a reformulation of the push-sum update that will be key to showing the exponential convergence of the algorithm. The proof is a minor variation of Lemma 4 in \cite{nedic2016stochastic}.
\begin{lemma}\label{lem: uAu}
	Consider the vectors $\bu(k)\in \R^d$, $\bv(k)\in \R^d_{+}$ and square matrix $\bA(k) \in \R^{d \times d}_{+}$, for $k\geq 0$ such that,
	\begin{align}\label{eq: uAu}
	\begin{split}
	\bu(k+1) = \bA(k)\bu(k),\\
	\bv(k+1) = \bA(k)\bv(k).
	\end{split}
	\end{align}
	Also suppose $u_i(k)=0$ if $v_i(k)=0$, $\forall k,i$.
	Define $\bu^-(k)\in \R^d$ as:
	\begin{align*}
	u_i^-(k):=\begin{cases}
	1/u_i(k), \qquad &\text{if } u_i(k)\neq 0,\\
	0,\qquad &\text{if } u_i(k)= 0.
	\end{cases}
	\end{align*}
	Define $\br(k):=\bu(k)\circ \bv^-(k)$, where $\circ$ denotes the element-wise product of two vectors. Then we have,
	\begin{align*}
	\br(k+1) =\bB(k)\br(k),
	\end{align*}
	where $\bB(k) \in \R^{d \times d}_{+}$ is defined as,
	\begin{align*}
	\bB(k):= {\rm diag}(\bv^-(k+1))\bA(k) {\rm diag}(\bv(k)).
	\end{align*}
\end{lemma}

\begin{proof}
	Since $u_i(k)=0$ if $v_i(k)=0$, $u_i(k)=r_i(k)v_i(k)$ holds for all $i,k$.
	Substituting in \eqref{eq: uAu} we obtain,
	\begin{align*}
	r_i(k+1)v_i(k+1) = \sum_{j=1}^{d}A_{ij}(k)r_j(k)v_j(k).
	\end{align*}
	Since, by definition $r_i(k)=0$ if $v_i(k)=0$, $\forall k,i$, we get
	\begin{align*}
	r_i(k+1) = v_i^-(k+1)\sum_{j=1}^{d}A_{ij}(k)r_j(k)v_j(k).
	\end{align*}
	Therefore,
	\begin{align*}
	\br(k+1)={\rm diag}(\bv^-(k+1))\bA(k) {\rm diag}(\bv(k))\br(k).
	\end{align*}
\end{proof}

Our next corollary, which follows immediately from the previous lemma, characterizes the dichotomy inherent in push-sum with virtual nodes: every row either adds up to one or zero.

\begin{corollary}\label{col: B stochastic}
	Consider the matrix $\bB(k)$ defined in Lemma \ref{lem: uAu}.
	Let us define the index set $J^k:=\{i |\, v_i(k)\neq 0 \}.$
	If $i\notin J^k$, the $i$th column of $\bB(k)$ and $i$th row of $\bB(k-1)$ only contain zero entries.
	Moreover,
	\begin{align*}
	\bB(k)\mathbf{1}_d&={\rm diag}(\bv^-(k+1))\bA(k)\bv(k)\\
	&={\rm diag}(\bv^-(k+1))\bv(k+1)=
	\begin{bmatrix}
	\text{$1$ or $0$}\\
	\vdots\\
	\text{$1$ or $0$}
	\end{bmatrix}.
	\end{align*}
	Hence, the $i$th row of $\bB(k)$ sums to $1$ if and only if $\bv_i(k+1)\neq 0$ or $i\in J^{k+1}$. 
\end{corollary}

Our next lemma characterizes the relationship between zero entries in the vectors $\bchi(k)$ and $\bpsi(k)$.

\begin{lemma} \label{lem: x=0 if y=0}
	$\chi_h(k)=0$ whenever $\psi_h(k)=0$ for $h=1,\ldots,n+m'$, $k\geq 0$.
\end{lemma}

\begin{proof}
	First we note that $\bpsi(0) = [\mathbf{1}_n^{\top},\mathbf{0}_{m'}^{\top} ]^{\top}$ and each node $i\in \V$ has a self-loop in graph $\H(k)$ for all $k\geq 0$; hence, $\psi_h(k)\geq 0$ for all $h$ and particularly, $\psi_i(k)> 0$ for $i=1,\ldots,n$.
	Now suppose $h>n$ and corresponds to a virtual agent $a_h\in \V_A$.
	If $\psi_h(k)=0$, it means $a_h$ has already sent all its $y$-value to another node or has not received any $y$-value yet. In either case, that node also has no remaining $x$-value as well and $\chi_h(k)=0$.
\end{proof}

Let us define $\bpsi^-(k)\in \R^{n+m'}$, $k \geq 0$ by
\begin{align}\label{eq: Y-}
\psi_i^-(k):=\begin{cases}
1/\psi_i(k), \qquad &\text{if } \psi_i(k)\neq 0,\\
0,\qquad &\text{if } \psi_i(k)= 0.
\end{cases}
\end{align}
Moreover, we define the vector $\bz(k)$ by setting $\bz(k):=\bchi(k) \circ \bpsi^-(k)$.
By \eqref{eq: XMX} and Lemma \ref{lem: x=0 if y=0}, we can use Lemma \ref{lem: uAu} to obtain,
\begin{align*}
\bz(k+1)=\bP(k)\bz(k),
\end{align*}
where $\bP(k):= {\rm diag}(\bpsi^-(k+1))\bM(k) {\rm diag}(\bpsi(k))$. Let us define $$I^k:=\{i | \, \psi_i(k)>0\}.$$ Then, by Corollary \ref{col: B stochastic} we have each $z_i(k+1)$, $i \in I^{k+1}$, is a convex combination of $z_j(k)$'s for $j\in I^k$. Therefore,
\begin{align}\label{eq: Z sandwich}
\begin{split}
\max_{i\in I^{k+1}}z_i(k+1)\leq \max_{i\in I^{k}}z_i(k),\\
\min_{i\in I^{k+1}}z_i(k+1)\geq \min_{i\in I^{k}}z_i(k).
\end{split}
\end{align}

These equations will be key to the analysis of the algorithm. We stress that we have not shown that the quantity $\min_{i} z_i(k)$ is nondecreasing; rather, we have shown that the related quantity, where the minimum is taken over $I^k$, the set of nonzero entries of $\bpsi(k)$, is nonincreasing.

Our next lemma provides lower and upper bounds on the entries of the vector $\bpsi(k)$.

\begin{lemma}\label{lem: Y lower bound}
	For $k\geq 0$ and $1\leq i\leq n$ we have:
	\begin{align*}
	n\alpha \leq \psi_i(k) \leq n. 
	\end{align*}
	Moreover, for $n+1\leq h\leq n+m'$ and $k\geq 1$ we have either $\psi_h(k)=0$ or,
	\begin{align*}
	n\alpha^2\leq \psi_h(k) \leq n.
	\end{align*}
\end{lemma}

\begin{proof}
	We have,
	\begin{align*}
	\bpsi(k)=\bM^{k-1:0}
	\begin{bmatrix}
	\mathbf{1}_n\\\mathbf{0}_{m'}
	\end{bmatrix},
	\end{align*}
	If $k<nL_s$, positive entries of $\bM^{k-1:0}$ are at least $(1/n)^k$. Hence, positive entries of $\bpsi(k)$ are at least,
	\begin{align*}
	\left(\frac{1}{n}\right)^k \geq \left(\frac{1}{n}\right)^{nL_s-1}=n\alpha.
	\end{align*}
	Now suppose $k\geq nL_s$. $\bM^{k-1:0}$ is the product of $\bM^{k-1:k-nL_s}$ and another column stochastic matrix. By Lemma \ref{lemma:positive n-rows}, $\bM^{k-1:k-nL_s}$ has positive first $n$ rows, and positive entries of at least $\alpha$. Thus, $\bM^{k-1:0}$ has positive first $n$ rows, and positive entries of at least $\alpha$ as well. We obtain for $1\leq i\leq n$,
	\begin{align*}
	\psi_i(k)\geq n \alpha,\text{ for } k\geq 1.
	\end{align*}
	For $n+1\leq h\leq n+m'$, suppose $\psi_h$ corresponds to a virtual node $a_h$ corresponding to some link $(i,j)\in \Ed$. If $\psi_h(k)$ is positive, it is carrying a value sent from $i$ at $k-nL_s$ or later, which has experienced link failure or delays. This is because each value gets to its destination after at most $L_s$ iterations. Since $i$ has self-loops all the time, $a_h$ is reachable from $i$ in period $k-nL_s$ to $k-1$; Hence, $M_{hi}^{k-1:k-nL_s}\geq \alpha$, and it follows,
	\begin{equation*}
	\psi_h(k)\geq\alpha \psi_i(k-nL_s)\geq n\alpha^2.
	\end{equation*}
	Also, due to sum preservation property, we have $\psi_h(k)\leq n$, for all $h$ and $k\geq 0$.
\end{proof}

Using Lemma \ref{lem: uAu} again, it follows,
\begin{align*}
\bz(k+nL_s)=\hat{\bP}(k)\bz(k),
\end{align*}
where,
\begin{align}\label{eq: tP def}
\hat{\bP}(k):= {\rm diag}(\bpsi^-(k+nL_s))\bM^{k+nLs-1:k} {\rm diag}(\bpsi(k)).
\end{align}
Next, we are able to find a lower bound on the positive elements of $\hat{\mathbf{P}}(k)$. The proof of the following corollary is immediate.

\begin{corollary}\label{col: P}
	By \eqref{eq: tP def} and Lemma \ref{lem: Y lower bound} we have:
	\begin{enumerate}
		\item[(a)] $\hat{P}_{ij}(k)>0$ for $1\leq i,j \leq n$.
		\item[(b)] Positive entries of first $n$ columns of $\hat{P}(k)$ are at least $(1/n)\alpha (n\alpha)=\alpha^2$. Similarly, the last $m'$ columns have positive entries of at least $\alpha^3$.
		\item[(c)] For $h>n$, if $h\in {I}^{k+nL_s}$ then $\hat{P}_{hi}(k)>0$ for some $1\leq i \leq n$.
	\end{enumerate}
\end{corollary}

Our next lemma, which is the final result we need before proving the exponential convergence rate of RAPS, provides a quantitative bound for how multiplication by the matrix $\bP$ shrinks the range of a vector.

\begin{lemma}\label{lem: s_t}
	Let $t\geq 0$ and $\{\bu(k)\}_{k \geq 0} \in \R^{n+m'}$ be a sequence of vectors such that,
	\begin{align*}
	\bu(k+1) = \hat{\bP}(knL_s+t)\bu(k).
	\end{align*}
	Define
	\begin{align*}
	s_t(k)&:=\max_{i \in I^{knL_s+t}}u_i(k) - \min_{i\in I^{knL_s+t}}u_i(k).
	\end{align*}
	Then,
	\begin{align*}
	s_t(k+2)\leq (1-n\alpha^6)s_t(k).
	\end{align*}
\end{lemma}

\begin{proof}
	Let us define
	\[ r_t(k):=\max_{1\leq i \leq n}u_i(k) - \min_{1\leq i \leq n}u_i(k).\]
	By Corollary \ref{col: P} for $j\in I^{(k+1)nL_s+t}$, the $j$th row of $\hat{\bP}(knL_s+t)$ has at least one positive entry in the first $n$ columns. Thus, because $u_j(k+1)$ is maximized/minimized when all of the weight is put on the largest/smallest possible entry of $u_j(k)$, we have:
	\begin{align*}
	u_j(k+1)\leq \alpha^3 \max_{1\leq i \leq n} u_i(k) + (1-\alpha^3) \max_{i \in I^{knL_s+t}} u_i(k),\\ 
	u_j(k+1) \geq \alpha^3 \min_{1\leq i \leq n} u_i(k) + (1-\alpha^3) \min_{i \in I^{knL_s+t}} u_i(k),
	\end{align*}
	Therefore,
	\begin{align}\label{eq: s_t}
	s_t(k+1)\leq \alpha^3 r_t(k)+ (1-\alpha^3)s_t(k).
	\end{align}
	Moreover, by a similar argument for $j \leq n$,
	\begin{align*}
	u_j(k+1)\leq \alpha^3 \sum_{i=1}^n u_i(k) + (1-n\alpha^3) \max_{i \in I^{knL_s+t}} u_i(k),\\ 
	u_j(k+1) \geq \alpha^3 \sum_{i=1}^n u_i(k) + (1-n\alpha^3) \min_{i \in I^{knL_s+t}} u_i(k).
	\end{align*}
	Thus,
	\[
	r_t(k+1)\leq (1-n\alpha^3)s_t(k).
	\]
	Combining with \eqref{eq: s_t} and noting that $r_t(k)\leq s_t(k)$ and $s_t(k+1)\leq s_t(k)$ we obtain,
	\begin{align*}
	s_t(k+2)&\leq \alpha^3(1-n\alpha^3)s_t(k)+(1-\alpha^3)s_t(k+1)\\
	&\leq \alpha^3(1-n\alpha^3)s_t(k)+(1-\alpha^3)s_t(k)\\
	&= (1-n\alpha^6)s_t(k).
	\end{align*}
\end{proof}

\begin{proof}\textbf{of Theorem \ref{The: geometric push-sum}}
	Using Lemma \ref{lem: s_t} with $t=0$ and $\bu(k)=\bz(knL_s)$ we get  $s_0(k) \leq (1-n \alpha^6)^{\lfloor k/2\rfloor}s_0(0)$ and $\lim_{k \rightarrow \infty}s_0(k)=0$. 
	Moreover by \eqref{eq: Z sandwich},  ${\bz}_{\max}(k)$ is a non-increasing sequence and by ${\bz}_{\min}(k)$, is non-decreasing. Thus,
	\begin{equation}
	\lim_{k\to \infty,\; h\in {I}^k }{\bz}_h(k) = L_\infty. \label{eq:lim rz}
	\end{equation}
	We have:
	{\small
		\begin{align*}
		L_\infty &= L_\infty \lim_{k \to \infty} \frac{\sum_{i=1}^{n+m'} \psi_i(k) }{\sum_{i=1}^{n+m'} \psi_i(k)}\\
		&= \lim_{k \to \infty}\Bigg(\frac{\sum_{i=1}^{n+m'} z_i(k)\psi_i(k) }{n} 
		+ \frac{\sum_{i=1}^{n+m'} (L_\infty-z_i(k))\psi_i(k)}{n}\Bigg)\\
		&= \lim_{k \to \infty}\left(\frac{\sum_{i=1}^{n+m'} \chi_i(k)}{n}
		+ \frac{\sum_{i=1}^{n+m'} (L_\infty-z_i(k))\psi_i(k)}{n}\right)\\
		&= \frac{\sum_{i=1}^n x_i(0)}{n}.
		\end{align*}
	}
	
	In the above, we used \eqref{eq: sum preservation} and \eqref{eq:lim rz}, the boundedness of $\psi_i(k)$, and the fact that  $\psi_i(k)=0$ for $i \notin I^k$.
	
	Finally, to show the exponential convergence rate, we go back to $s_0(k)$. We have for $ k\geq 1$,
	\begin{align*}
	s_0(k) &\leq (1-n \alpha^6)^{\lfloor k/2\rfloor}s_0(0) \leq (1-n \alpha^6)^{(k-1)/2} s_0(0),\\ 
	s_0(0)&\leq \sum_{i=1}^{n+m'} \vert z_i(0)\vert = \sum_{i=1}^n \vert x_i(0)\vert = \Vert \bx(0) \Vert_1,
	\end{align*}
	where the first equality holds because ${I}^0=\{1,\ldots,n \}$ and $y_i(0)=1$.
	Therefore, we have for $i\in I^k$,
	\begin{align*}
	\left| z_i(k)-\frac{\mathbf{1}^{\top}\bx(0)}{n}\right| &\leq 
	z_{\max}(k)-z_{\min}(k) \\ 
	&\leq s_0(\lfloor k/nL_s \rfloor) \\
	&\leq (1-n\alpha^6)^{(\lfloor \frac{k}{nL_s} \rfloor -1)/2} \Vert \bx(0)\Vert_1\\
	&\leq (1-n\alpha^6)^{(\frac{k}{nL_s}-1-1)/2} \Vert \bx(0)\Vert_1\\
	&=\frac{1}{1-n\alpha^6} \left((1-n\alpha^6)^{1/(2nL_s)}\right)^k \Vert \bx(0)\Vert_1\\
	&=\delta \lambda^k \Vert \bx(0) \Vert_1.
	\end{align*}
	where $\delta=\frac{1}{1-n\alpha^6}$ and $\lambda=(1-n\alpha^6)^{1/(2nL_s)}$. Note that $\{1,\ldots,n\}\subseteq I^k,\, \forall k$.
\end{proof}

\bigskip

\noindent {\bf Remark:} Observe that our proof did not really use the initialization $\bpsi(0) = {\bf 1}$, except to observe that the elements $\bpsi(0)$ are positive, add up to $n$, and the implication that $\bpsi(k)$ satisfies the bounds of Lemma \ref{lem: Y lower bound}. In particular, the same result would hold if we viewed time $1$ as the initial point of the algorithm (so that $\bpsi(1)$ is the initialization), or similarly any time $k$. We will use this observation in the next subsection.

\subsection{Perturbed Push-sum}
In this subsection, we begin by introducing the Perturbed Robust Asynchronous Push-Sum algorithm, obtained by adding a perturbation to the $x$-values of (non-virtual) agents at the beginning of every iteration they wake up.

\begin{algorithm}[h]
	\caption{Perturbed Robust Asynchronous Push-Sum}
	\begin{algorithmic}[1]\label{alg: perturbed}
		\STATE Initialize the algorithm with $\by(0)=\mathbf{1}$, $\phi_i(0)=0$,  $\forall i\in\{1,\ldots,n\}$ and $\rho_{ij}(0)=0$, $\kappa_{ij}(0)=0$, $\forall (j,i)\in \Ed$ and $\bDelta(0)=\mathbf{0}$.
		\STATE At every iteration $k=0,1,2,\ldots$, for every node $i$:
		\IF{node $i$ wakes up}
		\STATE $x_i \leftarrow x_i + \Delta_i(k)$;
		\STATE Lines \ref{line4_alg1} to \ref{line17_alg1} of Algorithm \ref{alg: raps}
		\ENDIF
		\STATE Other variables remain unchanged.
	\end{algorithmic}
\end{algorithm}

We show that, if the perturbations are bounded, the resulting $\bz(k)$ nevertheless tracks the average of $\bchi(k)$ pretty well. Such a result is a key step towards analyzing distributed optimization protocols. In this general approach to the analyses of distributed optimization methods, we follow \cite{ram2010distributed} where it was first adopted; see also \cite{nedic2016stochastic} and \cite{nedic2015distributed} where it was used.

Adopting the notations introduced earlier and by the linear formulation \eqref{eq: XMX} we have,
\begin{align*}
\bchi(k+1)=\bM(k)( \bchi(k)+\bDelta(k)), \qquad \text{for } k\geq 0,
\end{align*}
where $\bDelta(k)\in\R^{n+m'}$ collects all perturbations $\Delta_i(k)$ in a column vector with $\Delta_h(k)\-:=\-0$ for $n<h\leq n+m'$. We may write this in a convenient form as follows.
\begin{align*}
\bchi(k+1) & = \bM(k)(\bchi(k)+\bDelta(k))\\
&= \sum_{t=1}^k \bM^{k:t} \bDelta(t) + \bM^{k:0}\bchi(0).
\end{align*}
Define for $k\geq 1$,
\begin{align}
\begin{split}
\bchi^t(k):=\bM^{k-1:t} \bDelta(t), \qquad &1 \leq t \leq k,\\
\bchi^0(k):=\bM^{k-1:0} \bchi(0),\qquad &t=0.
\end{split}
\end{align}
We obtain,
\begin{align}
\bchi(k)=\sum_{t=0}^{k-1} \bchi^t(k),\qquad k\geq 1.
\end{align}
Define $\bz^t(k):=\bchi^t(k)\circ \bpsi^-(k)$ for $0 \leq t \leq k$ (cf. \eqref{eq: Y-}). We have
\begin{align}\label{eq: z zt}
\bz(k)=\sum_{t=0}^{k-1} \bz^t(k).
\end{align}

We may view each $\bz^t(k)$ as the outcome of a push-sum algorithm, initialized at time $t$, and apply Theorem \ref{The: geometric push-sum}. This immediately yields the following result, with part (b) an immediate consequence of part (a).

\begin{theorem}\label{The: z delta converegence general}
	Suppose Assumption \ref{asm: connvectivity} holds. Consider the sequence $\{z_i(k)\}$, $1\leq i \leq n,$ generated by Algorithm \ref{alg: perturbed}. Then,
	\begin{enumerate}
		\item[(a)]
		For $k=1,2,\ldots$ \begin{small}
			\begin{align*} 
			\left|  z_i(k)-\frac{\mathbf{1}^{\top}\bchi(k)}{n}\right| &\leq \delta \lambda^k\Vert\bx(0)\Vert_1 + \sum_{t=1}^{k-1} \delta \lambda^{k-t} \Vert \bDelta(t)\Vert_1. 
			\end{align*} \end{small}
		\item[(b)]
		If $\lim_{t\rightarrow \infty}\Vert\bDelta(t)\Vert_1=0$ then,
		\begin{align*}
		\lim_{k\rightarrow \infty}\left| z_i(k)-\frac{\mathbf{1}^{\top}\bchi(k)}{n}\right|=0.
		\end{align*}
	\end{enumerate}
\end{theorem}

\section{Robust Asynchronous Stochastic Gradient-Push (RASGP)}\label{sec: optimization}

In this section we present the main contribution of this paper, a distributed stochastic gradient method with asymptotically network-independent and optimal performance over directed graphs which is robust to asynchrony, delays, and link failures.

Recall that we are considering a network $\G$ of $n$ agents whose goal is to cooperatively solve the following minimization problem
\begin{align*}
\text{minimize } F(\bz):=\sum_{i=1}^n f_i(\bz), \qquad \text{over } \bz\in \mathbb{R}^{d},
\end{align*}
where each $f_i:\mathbb{R}^{d}\rightarrow \R$ is a strongly convex function only known to agent $i$. We assume  agent $i$ has the ability to obtain noisy gradients of the function $f_i$. 

The RASGP algorithm is given as Algorithm \ref{alg: optimization}. Note that we use the notation $\hbg_i(k)$ for a noisy gradient of the function $f_i(\bz)$ at $\bz_i(k)$ i.e.,
\begin{align*}
\hbg_i(k) = \bg_i(k) + \be_i,
\end{align*}
where $\bg_i(k):= \nabla f_i(\bz_i(k))$ and $\be_i$ is a random vector.

The RASGP is based on a standard idea of mixing consensus and gradient steps, first analyzed in \cite{nedic2009distributed}. The push-sum scheme of Section \ref{sec: push-sum}, inspired by \cite{hadjicostis2016robust}, is used instead of the consensus scheme, which allows us to handle delays, asynchronicity, and message losses; this is similar to the approach taken in \cite{nedic2015distributed}. We note that a new step-size strategy is used to handle  asynchronicity: when a node wakes up, it takes steps with a step-size proportional to the sum of all the step-sizes during the period it slept. As far as we are aware, this idea is new.

\begin{algorithm}[h]
	\caption{Robust Asynchronous Stochastic Gradient-Push (RASGP)}
	\begin{algorithmic}[1]\label{alg: optimization}
		\setcounter{ALC@unique}{0}
		\STATE Initialize the algorithm with $\by(0)=\mathbf{1}$, $\boldsymbol \phi_i^x(0)=\mathbf{0}$, $\phi_i^y(0)=0$, $\kappa_i(0)=-1$,  $\forall i\in\{1,\ldots,n\}$ and $\boldsymbol \rho_{ij}^x(0)=\mathbf{0}$, $\rho_{ij}^y(0)=0$, $\kappa_{ij}(0)=-1$, $\forall (j,i)\in \Ed$.
		\STATE At every iteration $k=0,1,2,\ldots$, for every node $i$:
		\IF{node $i$ wakes up}
		\STATE $\beta_i(k)=\sum_{t=\kappa_i+1}^k \alpha(t)$; \label{betadef}
		\STATE $\bx_i \leftarrow \bx_i -\beta_i(k)\hbg_i(k)$;
		\STATE $\kappa_i \leftarrow k$;
		\STATE $\boldsymbol \phi_i^x \leftarrow \boldsymbol \phi_i^x +\frac{\bx_i}{d_i^{+}+1}$,
		$\phi_i^y \leftarrow \phi_i^y +\frac{y_i}{d_i^{+}+1}$;
		\STATE $\bx_i \leftarrow \frac{\bx_i}{d_i^{+}+1}$,
		$y_i \leftarrow \frac{y_i}{d_i^{+}+1}$;
		\STATE Node $i$ broadcasts $(\boldsymbol \phi_i^x, \phi_i^y, \kappa_i)$ to its out-neighbors: $N_i^{+}$
		\STATE \textbf{Processing the received messages}
		\FOR{$(\boldsymbol \phi_j^x, \phi_j^y,\kappa_j')$ in the inbox}
		\IF{$\kappa_j'>\kappa_{ij}$} 
		\STATE $\boldsymbol \rho_{ij}^{*x} \leftarrow \boldsymbol \phi_j^x$,
		$\rho_{ij}^{*y} \leftarrow \phi_j^y$;
		\STATE $\kappa_{ij} \leftarrow \kappa_j'$;
		\ENDIF
		\ENDFOR
		\STATE $\bx_i \leftarrow \bx_i + \sum_{j\in N_i^{-}}\left(\boldsymbol \rho_{ij}^{*x}- \boldsymbol \rho_{ij}^x\right)$,
		$y_i \leftarrow y_i + \sum_{j\in N_i^{-}}\left(\rho_{ij}^{*y}-\rho_{ij}^y\right)$;
		\STATE $\boldsymbol \rho_{ij}^x \leftarrow \boldsymbol \rho_{ij}^{*x}$,
		$\rho_{ij}^y \leftarrow \rho_{ij}^{*y}$;
		\STATE $\bz_i \leftarrow \frac{\bx_i}{y_i}$;
		\ENDIF
		\STATE Other variables remain unchanged.
	\end{algorithmic}
\end{algorithm}

We will be making the following assumption on the noise vectors. 

\begin{assumption}\label{asm: bounded noise}
	$\be_i$ is an independent random vector with bounded support, i.e., $\Vert \be_i \Vert \leq b_i$, $i=1,\ldots,n$. Moreover, $\E[ \be_i ]=\mathbf{0}$ and $\E[\Vert \be_i \Vert^2] \leq \sigma_i^2$.
\end{assumption}

Next, we state and prove the main result of this paper, which states the linear convergence rate of Algorithm \ref{alg: optimization}.
\begin{theorem}\label{The: Opt conv}
	Suppose that:
	\begin{enumerate}
		\item
		Assumptions \ref{asm: connvectivity} and \ref{asm: bounded noise} hold.
		\item
		Each objective function $f_i(\bz)$ is $\mu_i$-strongly convex over $\mathbb{R}^{d}$.
		\item
		The gradients of each $f_i(\bz)$ are $L_i$-Lipschitz continuous, i.e., for all $\bz_1, \bz_2 \in \mathbb{R}^{d}$,
		\begin{align*}
		\Vert \bg_i(\bz_1) - \bg_i(\bz_2) \Vert \leq L_i \Vert \bz_1 - \bz_2 \Vert.
		\end{align*}
	\end{enumerate} 
	Then, the RASGP algorithm with the step-size $\alpha(k)=n/(\mu k)$ for $k\geq 1$ and $\alpha(0)=0$, will converge to the unique optimum $\bz^*$ with the following asymptotic rate: for all $i = 1, \ldots, n,$ we have

	\begin{align*}
	\E \left[ \Vert \bz_i(k)-\bz^* \Vert^2 \right]  \leq  \frac{L_u\sigma^2}{k \mu^2} + \O_k \left( \frac{1}{k^{1.5}} \right),
	\end{align*} where $\sigma^2 := \sum_i \sigma_i^2$, $\mu = \sum_i \mu_i$.
\end{theorem}

\as{
	\begin{remark}
	We note that each agent stores variables $\bx_i, y_i, \kappa_i, \bz_i, \bphi_i^x, \phi_i^y$ and $\boldsymbol \rho_{ij}^x, \rho_{ij}^y, \kappa_{ij}$ for all in-neighbors $j \in N_i^-$. Hence, the memory requirement of the RASGP algorithm for each agent is $\O(d_i^-)$ for each agent $i$.
\end{remark}
}

We next turn to the proof of Theorem \ref{The: Opt conv}.
First, we observe that Algorithm \ref{alg: optimization} is a specific case of multi-dimensional Perturbed Robust Asynchronous Push-Sum. 
In other words, each coordinate of vectors $\bx_i$, $\bz_i$, $\boldsymbol \phi_i^x$ and $\boldsymbol \rho_{ij}^x$ will experience an instance of Algorithm \ref{alg: perturbed}.
Hence, there exists an augmented graph sequence $\{\H(k)\}$ where the Algorithm \ref{alg: optimization} is equivalent to perturbed push-sum consensus on $\H(k)$ where each agent $a_h\in \V_A$ holds vectors $\bx_h$ and $y_h$. In other words, we will be able to apply Theorem \ref{The: z delta converegence general} to analyze Algorithm \ref{alg: optimization}.

Our first step is to show how to decouple the action of Algorithm \ref{alg: optimization} coordinate by coordinate. For each coordinate $1 \leq \ell \leq d$, let $\bchi^\ell\in \R^{n+m'}$ stack up the $\ell$th entries of $x$-values of all agents (virtual and non-virtual) in $\V_A$.
Additionally, define $\bDelta^\ell(k)\in \R^{n+m'}$ to be the vector stacking up the $\ell$th entries of perturbations. i.e.,
\begin{align*}
[\bDelta^\ell(k)]_i := \begin{cases}
-\beta_i(k)[\hbg_i(k)]_\ell,\qquad &\text{if } i \in \V, \tau_i(k)=1,\\
0, &\text{otherwise.}
\end{cases}
\end{align*}
Then, by the definition of the algorithm, we have for all $\ell=1,\ldots,d$,
\begin{align}\label{eq: linear form alg 3}
\begin{split}
\bchi^\ell(k+1) &= \bM(k)\left( \bchi^\ell(k) + \bDelta^\ell(k) \right),\\
\bpsi(k+1) &= \bM(k)\bpsi(k).
\end{split}
\end{align}
These equations write out the action of Algorithm \ref{alg: optimization} on a coordinate-by-coo\-rdi\-nate basis.

In order to prove Theorem \ref{The: Opt conv}, we need a few tools and lemmas. As already mentioned, our first step will be to argue that Algorithm \ref{alg: optimization} converges by application of Theorem \ref{The: z delta converegence general}. This requires showing the boundedness of the perturbations $\bDelta^{\ell}(k)$, which, as we will show, reduces to showing the vectors $\bz_i(k)$ are bounded. The following lemma will be useful to establish this boundedness.

\begin{lemma}\cite[Lemma 3]{nedic2016stochastic} \label{lem: u<v}
	Let $q:\R^{d}\rightarrow \R$ be a $\nu$-strongly convex function with $\nu>0$ which has Lipschitz gradients with constant $L$. let $\bv\in \R^{d}$ and $\bu \in \R^{d}$ defined by,
	\[\bu = \bv - \alpha(\nabla q(\bv)+ p(\bv)),\]
	where $\alpha \in \left(0,\nu/8L^2\right]$ and $p: \R^{d}\rightarrow \R^{d}$ is a mapping such that,
	\[ \Vert p(\bv) \Vert \leq c, \qquad \text{ for all } \bv\in \R^{d}.
	\]
	Then, there exists a compact set $\mathcal{S}\subset \R^{d}$ and a scalar $R$ such that,
	\[ \Vert \bu \Vert \leq \begin{cases}
	\Vert \bv \Vert, \qquad &\text{for all } \bv \notin \S,\\
	R, \qquad &\text{for all } \bv \in \S,
	\end{cases}
	\]
	where, \[\S:=\{\bz | \, q(\bz)\leq q(\mathbf{0}) + 2\frac{\nu}{8L^2} \left(\Vert q(\mathbf{0}) \Vert^2 + c^2\right)\} \cup B \left(\mathbf{0},\frac{4c}{\nu} \right),\]
	\[ R:=\max_{\bz\in \S} \{ \Vert \bz \Vert + \frac{\nu}{8L^2}\Vert \nabla q(\bz) \Vert \} + \frac{\nu c}{8L^2}. \]
\end{lemma}

We now argue that the iterates generated by Algorithm \ref{alg: optimization} are bounded.

\begin{lemma} \label{boundlemma} 
	The iterates $\bz_i(k)$ generated by Algorithm \ref{alg: optimization} will remain bounded.
\end{lemma}

\begin{proof}
	Let us adopt the notation $\bpsi^-$ from previous sections and define $\bz^\ell(k) :=\bchi^\ell(k)\circ \bpsi^-(k) \in \R^{n+m'}$. Moreover, adopt the notation $\bz_h$ for virtual agent $a_h$, $h=n+1,\ldots,n+m'$, as $\bz_h(k) := \bx_h(k)/\psi_h(k)$.
	Also define $\bu^\ell\in \R^{n+m'}$ by
	\begin{align*}
	\bu^\ell(k):=\bchi^\ell(k) + \bDelta^\ell(k).
	\end{align*}
	Since the perturbations are only added to the non-virtual agents, which have strictly positive $y$-values, we conclude $[u^\ell(k)]_h=0$ if $\psi_h(k)=0$. Hence, the assumptions of Lemma \ref{lem: uAu} and Corollary \ref{col: B stochastic} are satisfied. Adopting the definition of $I^k$ and $\bP(k)$ from previous sections, we get for $i\in I^{k+1}$,
	\begin{align*}
	[z^\ell(k+1)]_i = \sum_{j \in I^k}P_{ij}(k) \frac{[u^\ell(k)]_j}{\psi_j(k)}.
	\end{align*}
	Combining the equation above for $\ell=1,\ldots,d$ we obtain:
	\begin{align}\label{eq: z convex combination}
	\bz_i(k+1)=\sum_{j \in I^k}P_{ij}(k)\frac{\bu_j(k)}{\psi_j(k)},
	\end{align}
	where $\bu_j(k)\in \R^{d}$ is created by collecting the $j$th entries of all $\bu^\ell(k)$, i.e., 
	\begin{align*}
	\bu_i(k)=\begin{cases}
	\bx_i(k)-\beta_i(k)\hbg_i(k),\quad &\text{if } i \in \V \text{ and } \tau_{i}(k)=1,\\
	\bx_i(k), &\text{otherwise}.
	\end{cases}
	\end{align*}
	Now consider each term on the right hand side of \eqref{eq: z convex combination} for $j\in I^k$. Suppose $j \leq n$ and $\tau_j(k)=1$, then we have:
	\begin{align*}
	\frac{\bu_j(k)}{y_j(k)}=\bz_j(k) - \frac{\beta_j(k)}{y_j(k)} (\nabla f_j(\bz_j(k))+\be_j(k)).
	\end{align*}
	Since $\lim_{k\rightarrow \infty}\alpha(k)=0$ and $k-\kappa_i(k)\leq L_u$, $\lim_{k\rightarrow \infty}\beta_j(k)=0$. Moreover, by Lemma \ref{lem: Y lower bound}, $y_j(k)$ is bounded below; thus, $\lim_{k\rightarrow \infty} \beta_j(k)/y_j(k) = 0$ and there exists $k_j$ such that for $k\geq k_j$, $\beta_j(k)/y_j(k)\in \left(0,\mu_j/8L_j^2\right]$. Applying Lemma \ref{lem: u<v}, it follows that for each $j$ there exists a compact set $\S_j$ and a scalar $R_j$ such that for $k\geq k_j$, if $\tau_j(k)=1$,
	\begin{align} \label{eq: x/y< z R}
	\left\Vert \frac{\bu_j(k)}{y_j(k)}\right\Vert \leq 
	\begin{cases}
	\Vert \bz_j(k) \Vert, \; &\text{if } \bz_j(k)\notin \S_j, \\
	R_j, &\text{if } \bz_j(k)\in \S_j.
	\end{cases}
	\end{align}
	Moreover, if $\tau_j(k)=0$ or $j>n$ we have,
	\begin{align}\label{eq: wy z}
	\frac{\bu_j(k)}{y_j(k)}=\bz_j(k).
	\end{align}
	Let $k_z := \max_i k_i$. Using mathematical induction, we will show that for all $k \geq k_z$:
	\begin{align}\label{eq: z<R}
	\max_{i \in I^k} \Vert \bz_i(k) \Vert &\leq \bar{R},
	\end{align}
	where $\bar{R}:=\max \{\max_i R_i,\max_{j \in I^{k_z}} \Vert \bz_j(k_z) \Vert\}$.
	Equation \eqref{eq: z<R} holds for $k=k_z$. Suppose it is true for some $k \geq k_z$. Then by \eqref{eq: x/y< z R} and \eqref{eq: wy z} we have,
	\begin{align}\label{eq: x/y < R}
	\left\Vert \frac{\bu_i(k)}{y_i(k)}\right\Vert \leq \max \{ R_i, \Vert \bz_i(k) \Vert \}\leq \bar{R}.
	\end{align}
	Also by \eqref{eq: z convex combination}, for $i\in I^{k+1}$, $\bz_{i}(k+1)$ is a convex combination of $\bu_j(k)/y_j(k)$'s, where $j\in I^k$. Hence,
	\begin{align*}
	\Vert \bz_i(k+1) \Vert \leq \sum_{j \in I^k}P_{ij}  \left\Vert\frac{\bu_j(k)}{\psi_j(k)} \right\Vert \leq \bar{R}.
	\end{align*}
	Define $B_z:=\max \{ \bar{R}, \max_{i\in I^k,k<k_z} \Vert \bz_i(k) \Vert \}$ and we have $\Vert \bz_i(k) \Vert \leq B_z$, $\forall k \geq 0$.
\end{proof}

We next explore a convenient way to rewrite Algorithm \ref{alg: optimization}. 
Let us introduce the quantity $\bw_i(k)$, which can be interpreted as the $x$-value of agent $i$, if it performed a gradient step at every iteration, even when asleep:
	\begin{align}
	\bw_i(k):=\begin{cases}
	\bx_i(k)-\left(\sum_{t=\kappa_i(k)+1}^{k-1}\alpha(t)\right)\bg_i(k),\qquad &\text{if } i \in \V,\\
	\bx_i(k), &\text{otherwise}.
	\end{cases} \label{wdef}
	\end{align}
Also, define $\bw^\ell\in \R^{n+m'}$ by collecting the $\ell$th dimension of all $\bw_i$'s and \linebreak[4] $\bar{\bw}(k):=(\sum_{i=1}^{n+m'}\bw_i(k))/n$.
Moreover, define $\bg^\ell \in \R^{n+m'}$ by collecting the $\ell$th value of gradients  of all agents ($0$ for virtual agents), i.e.,
\begin{align*}
[\bg^\ell(k)]_i=\begin{cases}
[\bg_i(k)]_\ell, \qquad &\text{if } i \in \V,\\
0, \qquad &\text{otherwise.}
\end{cases}
\end{align*}
Additionally, define $\hat \be_i(k) \in \R^{d}$ as the noise injected to the system at time $k$ by agent $i$, i.e.,
\begin{align*}
	\hat \be_i(k) = \begin{cases}
	\beta_i(k) \be_i(k), \qquad &\text{if } i \in \V \text{ and } \tau_i(k)=1,\\
	\bf 0, \qquad &\text{otherwise,}
	\end{cases}
\end{align*}
and $\hat \be^\ell(k) \in \R^{n+m'}$ as the vector collecting the $\ell$th values of all $\hat \be_i(k)$'s.

We then have the following lemma.
\begin{lemma}
	\begin{align}\label{eq: W update}
	\bw^\ell(k+1)=\bM(k)\left(\bw^\ell(k)-\alpha(k)\bg^\ell(k) - \hat \be^\ell \right).
	\end{align}
\end{lemma}
\begin{proof} We consider two cases:
	\begin{itemize}
		\item If $\tau_i(k)=0$, then \eqref{eq: W update}  reduces to $\bw_i(k+1) = \bw_i(k) - \alpha(k) \bg_i(k)$; noting that, because node $i$ did not update at time $k$ we have that  $\bg_i(k) = \bg_i(k+1)$ and this is the correct update. 
		\item For all other nodes (i.e., for both virtual nodes and nodes with $\tau_i(k)=1$), we have $[\bw^\ell(k) - \alpha(k) \hbg^\ell(k) - \hat \be^\ell(k)]_i = [\bchi^\ell(k) + \Delta^\ell(k)]_i$ in \eqref{eq: linear form alg 3}. Since $\bchi^\ell(k+1) = \bM(k) (\bchi^\ell(k) + \bDelta^\ell(k))$ and, using the definition of $\bw_i(k)$, we have that for these nodes,
		$$\bw_i(k+1) = \bx_i(k+1);$$ \eqref{eq: linear form alg 3} implies the conclusion.
	\end{itemize}
\end{proof}

This lemma allows us to straightforwardly analyze how the average of $\bw(k)$ evolves. Indeed,  
summing all the elements of \eqref{eq: W update} and dividing by $n$ for each $\ell=1,\ldots,d$ we obtain,

\begin{align}\label{eq: linear form}
\bar{\bw}(k+1) &=\bar{\bw}(k)-\frac{\alpha (k)}{n}\sum_{i=1}^{n}\bg_i(k) - \frac{1}{n}\sum_{i=1}^n \hat \be_i(k) \nonumber\\
&=\bar{\bw}(k) -\frac{\alpha(k)}{n}\sum_{i=1}^{n}\nabla f_i(\bar{\bw}(k)) - \frac{1}{n}\sum_{i=1}^n \hat \be_i(k) -\frac{\alpha(k)}{n}\sum_{i=1}^{n}\left(\bg_i(k)-\nabla f_i(\bar{\bw}(k))\right).
\end{align}

We next give a sequence of lemmas to the effect that all the quantities generated by the algorithm are close to each other over time.
Define,
\begin{align*}
\bar \bx(k) = \frac{1}{n} \sum_{a_h \in \V_A} \bx_h(k).
\end{align*}
where, recall, $\V_A$ is our notation for all the nodes in the augmented graph (i.e., including virtual nodes). Moreover, we will extend the definition of $\beta_i(k)$ from Line \ref{betadef} of Algorithm \ref{alg: optimization} to {\em all} $k$ via the same formula  $\beta_i(k):=\sum_{t=\kappa_i(k)+1}^k \alpha(t)$.
Our first lemma will show that each $\bz_i(k)$ closely tracks $\bar \bx(k)$. 
\begin{lemma} \label{lem: z-x decay}
	Using Algorithm \ref{alg: optimization} with $\alpha(k) = n/(k \mu)$, under the assumptions of Theorem \ref{The: Opt conv}, we have for each $i$, $\Vert \bz_i(k+1) -\bar{\bx}(k+1) \Vert = \O_k(1/k)$.
\end{lemma}

\begin{proof}
	By Theorem \ref{The: z delta converegence general}(a) we have for each $\ell$,
	\begin{align*}
	\left| [\bz^\ell(k+1)]_i-\frac{\mathbf{1}^{\top}  \bchi^\ell(k+1)}{n}\right| &\leq \delta \lambda^k\Vert  \bchi^\ell(0)\Vert_1 + \sum_{t=1}^k \delta \lambda^{k-t} \Vert \bDelta^\ell(t)\Vert_1.
	\end{align*}
	Summing the above inequality  for $\ell = 1, \ldots, d$ we obtain,
	\begin{align*}
	\Vert   \bz_{i}(k+1) -\bar\bx(k+1) \Vert_1 & \leq \sum_{j=1}^n \Big( \delta \lambda^k \Vert \bx_j(0)\Vert_1 + \sum_{t=1}^k \delta \lambda^{k-t}\beta_i(t)\tau_i(t) \Vert \hbg_j(t)\Vert_1  \Big).
	\end{align*}
	Moreover, 
	\begin{align}
	\beta_i(k)&=\sum_{t=\kappa_i(k)+1}^k \frac{n}{\mu t}\leq \frac{n}{\mu} \left(\frac{k-\kappa_i(k)}{\kappa_i(k)+1}\right). \label{betaprelim}
	\end{align}
	But 
	\begin{align*}
	\kappa_i(k) < k \leq \kappa_i(k) + L_u.
	\end{align*}
	Since $L_u\geq 1$, we obtain \[ k \leq (\kappa_i(k) +1)L_u,\]
	or, \[\frac{1}{\kappa_i(k)+1} \leq \frac{L_u}{k}.\]
	Thus, from \eqref{betaprelim} we have,
	\begin{align}
	\beta_i(k) \leq \frac{nL_u^2}{\mu k}. \label{eq: beta bound}
	\end{align}
	
	Define \begin{equation} \label{mdef} M_j := \max_{ \| \bz \| \leq B_z } \| \bg_j (\bz) \|_1, \end{equation} and observe that $M_j$ is finite by Lemma \ref{boundlemma}.   Also $\tau_j(k)\leq 1$. We obtain, 
	\begin{align*}
	\Vert   \bz_{i}(k+1) -\bar\bx(k+1) \Vert_1 &\leq \sum_{j=1}^n \Big( \delta \lambda^k \Vert \bx_j(0)\Vert_1 + \sum_{t=1}^k \delta\lambda^{k-t} \frac{nL_u^2}{\mu t} (M_j+b_j) \Big).
	\end{align*}
	Let $RHS$ denote the right hand side of the relation above. We have,
	\begin{align*}
	RHS &= \sum_{j=1}^n \Bigg( \delta \lambda^k \Vert \bx_j(0)\Vert_1  +  \frac{\delta nL_u^2}{\mu}(M_j+b_j)\Big( \sum_{t=1}^{\lfloor \frac{k}{2} \rfloor} \frac{\lambda^{k-t}}{t} + \sum_{t=\lfloor \frac{k}{2} \rfloor + 1}^{k} \frac{\lambda^{k-t}}{t} \Big) \Bigg)\\
	& \leq \sum_{j=1}^n \Bigg( \delta \lambda^k \Vert \bx_j(0)\Vert_1 +  \frac{\delta n L_u^2}{\mu}(M_j+b_j)\left( \frac{k}{2}\lambda^{\frac{k}{2}} + \frac{2}{(1-\lambda)k} \right) \Bigg) = \O_k \left(\frac{1}{k} \right),
	\end{align*}
	where we used the following relations,
	\begin{gather*}
	\sum_{t=1}^{\lfloor \frac{k}{2} \rfloor} \frac{\lambda^{k-t}}{t} \leq \lfloor \frac{k}{2} \rfloor \lambda^{k-\lfloor \frac{k}{2} \rfloor} \leq \frac{k}{2}\lambda^{\frac{k}{2}},\\
	\sum_{t=\lfloor \frac{k}{2} \rfloor + 1}^{k} \frac{\lambda^{k-t}}{t} \leq \sum_{t=0}^{\lceil \frac{k}{2} \rceil -1} \frac{\lambda^{t}}{\lfloor \frac{k}{2} \rfloor + 1} \leq \frac{2}{(1-\lambda)k}.	
	\end{gather*}
	Finally, $\Vert \bv \Vert_2 \leq \Vert \bv \Vert_1$ for all vectors $\bv$, completes the proof.
\end{proof}

An immediate consequence of this lemma is that the quantities $\bar \bx(k)$ and $\bar \bw(k)$ are close to each other.

\begin{lemma} \label{lem: x-w decay} 	Using Algorithm \ref{alg: optimization} with $\alpha(k) = n/(k \mu)$, under the assumptions of Theorem \ref{The: Opt conv}, we have, $\Vert \bar{\bx}(k)-\bar{\bw}(k) \Vert = \O_k(1/k)$.
\end{lemma}
\begin{proof}
	By definition of $\bar{\bw}$ we have,
	\begin{align*}
	\bar{\bx}(k) - \bar{\bw}(k) = \frac{1}{n} \sum_{i=1}^{n}\left(\sum_{t=\kappa_i(k)+1}^{k-1}\alpha(t)\right)\bg_i(k).
	\end{align*}
	Using \eqref{eq: beta bound} we have,
	\begin{align*}
	\Vert \bar{\bx}(k)-\bar{\bw}(k) \Vert \leq \frac{1}{n} \sum_{i=1}^{n} \beta_i(k) M_i 
	\leq \sum_{i=1}^{n} \frac{L_u^2 M_i}{n\mu k} = \O_k \left(\frac{1}{k}\right),
	\end{align*}
	where $M_i$ was defined through \eqref{mdef}.
\end{proof}

We next remark on a couple of implications of the past series of lemmas.
\begin{corollary}\label{col: z-w decay}
	We have $\Vert \bz_i(k)-\bar{\bw}(k) \Vert = \O_k(1/k)$.
\end{corollary}
\begin{lemma}\label{lem: g-f decay}
	$\Vert \bg_i(k)-\nabla f_i(\bar{\bw}(k)) \Vert = \O_k(1/k)$.
\end{lemma}
\begin{proof}
	Since $\nabla f_i$ is $L_i$-Lipschitz, we have,
	\begin{align*}
	\Vert \bg_i(k)-\nabla f_i(\bar{\bw}(k)) \Vert \leq L_i \Vert \bz_i(k)-\bar{\bw}(k) \Vert.
	\end{align*}
	Using Corollary \ref{col: z-w decay}, the lemma is proved.
\end{proof}

We are now in a position to rewrite Algorithm \ref{alg: optimization} as a sort of perturbed gradient descent.
Let us define,
\[\boldsymbol \eta(k):= \frac{1}{\mu k}\sum_{i=1}^n \left(\bg_i(k)-\nabla f_i(\bar{\bw}(k))\right).
\]
By Lemma \ref{lem: g-f decay}, $\boldsymbol \eta(k) = \O_k(1/k^2)$. Therefore, there exists $B_\eta$ such that $\boldsymbol \eta(k) \leq B_\eta/k^2$ for all $k\geq 1$.

By \eqref{eq: linear form} we have,
\begin{align}\label{mainupdate}
\bbw(k+1) = \bbw(k) -  \frac{1}{\mu k} \nabla F(\bbw(k)) -  \bar \be(k) - \boldsymbol \eta(k),
\end{align}
where 
\begin{itemize}
	\item The function $F:= \sum_{i=1}^n f_i \in \R^{d} \rightarrow \R$ is $\mu$-strongly-convex with $L$-Lipschitz gradient, where $L := \sum_{i=1}^n L_i$. 
	\item The noise $\bar \be(k):= (\sum_{i=1}^n \hat \be_i(k))/n$ is bounded (i.e., $\bar \be(k) \in B(0,r_e)$, with probability one, where $r_e:=(L_u/\mu)\sum_j b_j$), and  $\mathbb{E}[\bar \be(k)]=\mathbf{0}$.
\end{itemize}

In other words, with the exception of the $\boldsymbol \eta(k)$ term, what we have is exactly a stochastic gradient descent method on the function $F(\cdot)$.

The following lemmas bound $\bbe(k)$. Let us define $\nu_i(k) = k - \kappa_i(k)$ as the number of iterations agent $i$ has skipped since it's last update. By Assumption \ref{asm: connvectivity}, $\nu_i(k) \leq L_u$.
	\begin{lemma}\label{lem: bound beta}
		We have $\beta_i(k) = \O_k(1/k)$, $\forall i$. Moreover,
		\begin{align*}
			\beta_i(k) \leq \frac{n\nu_i(k)}{\mu k} + \O_k(k^{-2}).
		\end{align*}
	\end{lemma}
\begin{proof}
	Since $\nu_i(k)\leq L_u, \forall i,$ we have for $\kappa_i(k)\geq 1$,
	\begin{align*}
		\beta_i(k) &= \sum_{t = \kappa_i(k) + 1}^{k} \frac{n}{\mu t} \leq \frac{n}{\mu} \ln \left(\frac{k}{\kappa_i(k)} \right) \leq \frac{n}{\mu} \ln \left(\frac{k}{k - \nu_i(k)} \right) \\
		& = \frac{n}{\mu} \ln \left(1+\frac{\nu_i(k)}{k - \nu_i(k)} \right)
		\leq \frac{n\nu_i(k)}{\mu(k - \nu_i(k))} = \frac{n\nu_i(k)}{\mu k} + \O_k(k^{-2}).
	\end{align*}
\end{proof}
\begin{corollary}\label{cor: bounded k e}
	$ \mu k \Vert \bar \be(k)\Vert$ is bounded.
\end{corollary}
\begin{lemma}\label{lem: bound bbe}
	There exists $B_\epsilon>0$ such that We have,
	$$ \E[\Vert \bbe(k) \Vert^2] \leq \frac{L_u^2}{\mu^2 k^2}\sigma^2 + \frac{B_\epsilon}{k^4}.$$
\end{lemma}
\begin{proof}
	Using Lemma \ref{lem: bound beta}, we have for $k>L_u$,
	\begin{align*}
		\E[\Vert \bbe(k) \Vert^2] &= \E [\Vert \frac{1}{n} \sum_{i=1}^n \beta_i(k) \be_i(k) \tau_i(k) \Vert^2 ] = \frac{1}{n^2} \sum_{i=1}^{n} \beta_i^2(k) \E[\Vert \be_i(k) \Vert^2] \\
		&\leq \frac{1}{n^2} \sum_{i=1}^{n} \beta_i^2(k) \sigma_i^2
		\leq \frac{L_u^2}{\mu^2 k^2}\sigma^2 + \O_k(k^{-4}),
	\end{align*}
	where the second equality is the result of the noise terms being independent and zero-mean.
\end{proof}

Our next observation is a technical lemma which is essentially a rephrasing of Lemma \ref{lem: u<v} above.

\begin{lemma} There exists a constant $B_w$ and time $k_w$ such that $\Vert\bbw(k)\Vert \leq B_w$ with probability one, for $k\geq k_w$. \label{lem: bounded x}
\end{lemma}

\begin{proof}
	We have
	\begin{align*}
	\bbw(k+1) = \bbw(k) -  \frac{1}{\mu k} \left[ \nabla F(\bbw(k)) + \mu k \left( \bar\be(k) + \boldsymbol \eta(k)\right)\right],
	\end{align*}
	where $\mu k\Vert \bar \be(k)+ \etab(k) \Vert$
	is bounded. Moreover, there exists $k_w$ such that for $k\geq k_w$, $ \frac{1}{\mu k} \in \left( 0, \mu/8L^2 \right]$. Therefore, by Lemma \ref{lem: u<v} there exists a compact set $\mathcal{S}_w$ and a scalar $R_w>0$ such that for $k\geq k_w$,
	\begin{align*}
	\Vert \bbw(k+1) \Vert\leq \begin{cases}
	\Vert \bbw(k) \Vert , \quad &\text{for } \bbw \notin \mathcal{S}_w, \\
	R_w, \quad &\text{for } \bbw \in \mathcal{S}_w.
	\end{cases}
	\end{align*}
	Therefore, setting $B_w:= \max \{R_w, \Vert \bbw(k_w) \Vert \}$ will complete the proof.
\end{proof}

As a consequence of this lemma, because 
$\Vert \boldsymbol \eta(k)\Vert _2 \leq B_\eta$, this lemma implies there is a constant $B_1$ such that for $k\geq k_w$,
\begin{align} \label{eq: x-x* <B'}
\left\Vert \bbw(k) -  \bz^* - \frac{1}{\mu k} \nabla F(\bbw(k)) -  \bar \be(k)  \right\Vert \leq B_1, 
\end{align}
with probability one. This now puts us in a position to show that $\bar \bw(k)$ converges in mean square to the optimal solution.

\begin{lemma} \label{lem: conv} $\E[\Vert \bbw(k) - \bz^* \Vert^2] \rightarrow 0$.
\end{lemma}

\begin{proof}
	Using the definition of $k_w$ from Lemma \ref{lem: bounded x}, we have that for  $k\geq k_w$, 
	\begin{small} 
		\begin{multline*} \E[\Vert \bbw(k+1) - \bz^*\Vert^2] \leq \E \Big[\Vert \bbw(k) - \bz^* - \frac{1}{\mu k} \nabla F(\bbw(k)) - \bbe(k) \Vert^2 \\ 
		+ 2 \Vert \boldsymbol \eta(k) \Vert \Vert \bbw(k)  -  \bz^* - \frac{1}{\mu k} \nabla F(\bbw(k)) - \bbe(k) \Vert + \Vert\boldsymbol \eta(k)\Vert^2 \Big].
		\end{multline*}
	\end{small} 
	We will bound each of the terms on the right. We begin with the easiest one, which is the last one:
	\begin{align}\label{eq: last}
	\Vert\boldsymbol \eta(k)\Vert^2 \leq \frac{B_\eta^2}{k^4}.
	\end{align}
	The middle term is bounded as 
	\begin{align} \label{eq: middle} 
	2\Vert \boldsymbol \eta(k) \Vert \Vert \bbw(k) -  \bz^* - \frac{1}{\mu k} \nabla F(\bbw(k)) - \bbe(k) \Vert \leq \frac{2B_\eta B_1}{k^2},
	\end{align}
	where we used \eqref{eq: x-x* <B'}. 
	
	Finally, we turn to the first term which we denote by $T_1$:
	\begin{multline*}
	T_1 \leq \E\Vert \bbw(k) - \bz^*\Vert^2 - \frac{2}{\mu  k} \E [ \nabla F(\bbw(k))^{\top} (\bbw(k) - \bz^*)] \\
	+ \frac{L^2}{\mu^2  k^2} \E[\Vert \bbw(k) - \bz^*\Vert^2] + \E[\Vert \bbe(k) \Vert^2],
	\end{multline*}
	where we used the usual inequality $\Vert \nabla F(\bbw(k))\Vert^2 \leq L^2 \Vert \bbw(k) - \bz^*\Vert^2$ which follows from $\nabla F (\cdot)$ being $L$-Lipschitz. Now, using the standard inequality 
	\begin{align*}
	\nabla F(\bbw(k))^T (\bbw(k) - \bz^*) &\geq F(\bbw(k)) - F(\bz^*) + \frac{\mu}{2} \Vert \bbw(k) - \bz^*\Vert^2 \\
	&\geq \mu \Vert \bbw(k) - \bz^*\Vert^2,
	\end{align*}
	and Lemma \ref{lem: bound bbe} we obtain,
	\begin{align} \label{eq: first} 
	T_1 \leq \left( 1 - \frac{2}{k} + \frac{L^2}{\mu^2  k^2} \right) \E[ \Vert \bbw(k) - \bz^*\Vert^2] + \frac{L_u^2}{\mu^2 k^2}\sigma^2 + \frac{B_\epsilon}{k^4}.
	\end{align} 
	Now putting together \eqref{eq: last}, \eqref{eq: middle}, and \eqref{eq: first}, we get,
	\begin{multline*}
	\E[\Vert \bbw(k+1) - \bz^*\Vert^2]  \leq \left( 1 - \frac{2}{k} + \frac{L^2}{\mu^2  k^2} \right) \E [\Vert \bbw(k) - x^*\Vert^2] +  \frac{L_u^2\sigma^2}{\mu^2  k^2} + \frac{2 B_\eta B_1}{ k^2}  + \frac{B_\eta^2 + B_\epsilon}{  k^4}.
	\end{multline*}
	For large enough $k$, we can bound the inequality above as,
	\begin{equation} \label{1overkeq}  \E[\Vert \bbw(k+1) - \bz^*\Vert^2]  \leq \left( 1 - \frac{1.5}{k}  \right)\E [ \Vert \bbw(k) - \bz^*\Vert^2] +   \frac{B_2}{ k^2}, \end{equation}
	where $B_2 = L_u^2\sigma^2/\mu^2 + 2 B_\eta B_1 + B_\eta^2 + B_\epsilon$.
	Using Lemma \ref{lem: decaying sequence}, stated next, we conclude $\E[\Vert \bbw(k) - \bz^* \Vert^2] \rightarrow 0$.
\end{proof}

\begin{lemma}\label{lem: decaying sequence}
	Let $a>1, b\geq 0$ and $\{x_t\}$ be a non-negative sequence which satisfies,
	\begin{align*}
	x_{t+1} \leq \left(1-\frac{a}{t} \right)x_t + \frac{b}{t^2},\qquad\text{for }t\geq t'>0.
	\end{align*}
	Then for all $t \geq t'$ we have, \[ x_t \leq \frac{m}{t},\]
	where $m:=\max \{t'x_{t'}, b/(a-1)\}$.
\end{lemma}

This lemma is stated and proved for $t'=1$ in \cite[Lemma 3]{rakhlin2012making}, and the case of general $t'$ follows immediately.

We are almost ready to complete the proof of Theorem \ref{The: Opt conv}; all that is needed is to refine the convergence rate of $\bar \bw(k)$ to $x^*$. 
Now as a consequence of \eqref{1overkeq}  and Lemma \ref{lem: decaying sequence}, we may use the inequality $\E[|X|] \leq \sqrt{\E[X^2]}$ to obtain that
\begin{equation} \label{1oversqrt} \E[\Vert \bbw(k) - \bz^*\Vert] = \O_k \left( \frac{1}{\sqrt{k}} \right).
\end{equation}
Furthermore, 
by the finite support of $\mu k \bar \be(k)$, by Corollary \ref{cor: bounded k e}, we also have that
\begin{align} \label{eq: conve}
\E [\Vert \bbw(k)  - \bz^* - \frac{1}{\mu  k} \nabla F(\bbw(k)) - \bbe(k) \Vert ]= \O_k \left( \frac{1}{\sqrt{k}} \right).
\end{align}  

We now use these observations to provide a proof of our main result. 

\begin{proof}\textbf{of Theorem \ref{The: Opt conv}}
	Essentially, we rewrite the proof of Lemma \ref{lem: conv}, but now using the fact that $\E[\Vert \bbw(k) - \bz^*\Vert] = \O_k(1/\sqrt{k})$ from \eqref{1oversqrt}. This allows us to make two modification to the arguments of that lemma. 
	First, we can now replace \eqref{eq: middle} by 
	\begin{small}
		\begin{align} \label{eq: middle2}
		\E [2\Vert \boldsymbol \eta(k) \Vert \Vert \bbw(k) -  \bz^* - \frac{1}{\mu k} \nabla F(\bbw(k)) - \bbe(k) \Vert] \leq \frac{2B_\eta}{ k^2} \O_k \left( \frac{1}{\sqrt{k}} \right),
		\end{align}
	\end{small}
	where we used \eqref{eq: conve}.
	Second,  putting together \eqref{eq: last}, \eqref{eq: middle2}, and \eqref{eq: first}, we obtain:
	\begin{multline*}
	\E[\Vert \bbw (k+1) - \bz^*\Vert^2]  \leq \left( 1 - \frac{2}{k} + \frac{L^2}{\mu^2  k^2} \right) \E [\Vert \bbw(k) - \bz^*\Vert^2] \\+ \E[\Vert \bbe(k) \Vert^2] + \frac{B_\eta^2}{ k^4} + \frac{2B_\eta}{  k^2} \O_k \left( \frac{1}{\sqrt{k}} \right).
	\end{multline*}
	which, again using the fact that $\E[\Vert \bbw(k) - \bz^*\Vert^2]= \O_k(1/\sqrt{k})$, we simply rewrite as, 
	\begin{align*}
	\E[\Vert \bbw (k+1) - \bz^*\Vert^2]  &\leq \left( 1 - \frac{2}{k} \right) \E [\Vert \bbw(k) - \bz^*\Vert^2] + \E[\Vert \bbe(k) \Vert^2] + \O_k \left(\frac{1}{k^{2.5}} \right).
	\end{align*}
	
	To save space, let us define $a_k := \E [\Vert \bbw(k) - \bz^*\Vert^2]$. Multiplying both sides of relation above by $k^2$ we obtain, 
	\begin{align*}
		a_{k+1}k^2 \leq a_k \left(1 - \frac{2}{k} \right)k^2 + \E[\Vert \bbe(k) \Vert^2]k^2 + \O_k(k^{-0.5}).
	\end{align*}
	Note that,
	\[ \left(1 - \frac{2}{k}\right)k^2 = k^2 - 2k < (k-1)^2. \]
	Thus,
	\begin{align*}
	a_{k+1}k^2 \leq a_k(k-1)^2 + \E[\Vert \bbe(k) \Vert^2]k^2 + \O_k(k^{-0.5}).
	\end{align*}
	Summing the relation above for $k=0,\ldots,T$ implies,
	\[ a_{T+1} T^2 \leq \sum_{k=0}^{T} \E[\Vert \bbe(k) \Vert^2]k^2 + \O_T(T^{0.5}). \]
	Now, let us estimate the first term on the right hand side of relation above,
	\begin{align*}
		 \sum_{k=0}^{T} \E[\Vert \bbe(k) \Vert^2]k^2 \leq  \sum_{k=0}^{T} \sum_{i=1}^n \frac{\beta_i^2(k)}{n^2}\sigma_i^2 \tau_i(k) k^2 = \sum_{i=1}^{n} \frac{\sigma_i^2}{\mu^2}\sum_{k=0}^{T} \nu_i(k)^2\tau_{i}(k) + \O_T(T^{-1}),
	\end{align*}
	where we used Lemma \ref{lem: bound beta} in the last equality. Define $t_i(j)$ as the $j$'th time agent $i$ has woken up, and set $t_i(0)=-1$. Then we can rewrite the relation above as,
	\begin{align*}
		\sum_{k=0}^{T} \nu_i(k)^2 \tau_i(k) =  \sum_{j=1}^{t_i(j)\leq T} (t_i(j) - t_i(j-1))^2 \leq \sum_{j=1}^{t_i(j)\leq T} L_u (t_i(j) - t_i(j-1))\leq L_u (T+1).
	\end{align*}
	Combining relations above and then dividing both sides by $T^2$ we obtain,
	\begin{align}
		a_{T+1} \leq \frac{L_u \sigma^2}{\mu^2 T} + \O_T(T^{-1.5}). \label{eq: wbardecay}
	\end{align}
	
	We next argue that the same guarantee holds for every $\bz_i(k)$. Indeed, for each $i=1, \ldots, m,$
		\begin{eqnarray*}
			\Vert\bz_i(k) - \bz^*\Vert^2   & = & \Vert\bz_i(k) - \bar \bw(k) + \bar \bw(k) - \bz^* \Vert^2 \\ 
			& = & \Vert\bz_i(k) - \bar \bw(k)\Vert^2  + 2 \Vert\bz_i(k) - \bar \bw(k) \Vert \Vert \bar \bw(k) - \bz^*\Vert + \Vert\bar \bw(k) - \bz^*\Vert^2.
		\end{eqnarray*}
	Now from Corollary \ref{col: z-w decay}, we know that with probability one, $\Vert\bz_i(k) - \bar \bw(k)\Vert_2 = \O_k(1/k)$. Taking expectation of both sides and using \eqref{eq: wbardecay} along with the usual  bound $\E[|X|] \leq \sqrt{\E[X^2]}$, we have
	\[ \E[\Vert\bz_i(k) - \bz^*\Vert^2] = \O_k \left( \frac{1}{k^2} \right) + \O_k\left( \frac{1}{k^{1.5}} \right) + \E[\Vert\bar \bw(k) - \bz^*\Vert^2]. \]
	Putting this together with \eqref{eq: wbardecay}  completes the proof.
\end{proof}

\subsection{Time-varying graphs}  We remark that Theorems \ref{The: geometric push-sum}, \ref{The: z delta converegence general} and \ref{The: Opt conv} all extend verbatim to the case of time-varying graphs with no message losses. Indeed, only one problem appears in extending the proofs in this paper to time-varying graphs: a node $i$ may send a message to node $j$; that message will be lost; and afterwards node $i$ never sends anything to node $j$ again. In this case, Lemmas \ref{lemma:positive n-rows} and \ref{lem: Y lower bound} do not hold. Indeed, examining Lemma \ref{lem: Y lower bound}, we observe what can very well happen is that all of $\chi_i(k)$ and $\psi_i(k)$ are ``lost'' over time into messages that never arrive. However, as long as no messages are lost, the proofs in this paper extend to the time-varying case verbatim.
\as{ On a technical level, the results still hold if $\bu_{ij}^x(k) = {\bf 0}, u_{ij}^y(k)=0$ (virtual node $c_{ij} \in \V_A$ holds no lost message), when link $(i,j)$ is removed from the network at time $k$, and the graph $\G$ stays strongly connected (or $B$-connected, i.e., there exists a positive integer $B$ such that the union of every $B$ consecutive graphs is strongly connected).}

\subsection{On the bounds for delays, asynchrony, and message losses} \as{It is natural to what extent the assumption of finite upper bounds on delays, asynchrony, and message losses are really necessary. A natural example which falls outside our framework is a fixed graph $G$, where, at each time step, every link in $G$ appears with probability $1/2$. A more general model might involve a different probability $p_e$ of failure for each edge $e$. }

\as{We observe that our result can already handle this case in the following manner. For simplicity, let us stick with the scenario where every link appears with probability $1/2$. Then the probability that, after time $t$, some link has not appeared is at most $m (1/2)^t$, where $m$ is the number of edges in $G$. This implies that if we choose $B=O(\log(mnT))$, then with high probability, the sequence of graphs $G_1, \ldots, G_T$ is $B$-connected.} 

\as{Thus our theorem applies to this case, albeit at the expense of some logarithmic factors due to the choice of $B$. We remark that it is possible to get rid of these factors by directly analyzing the decrease in $E[||z(t) - z^*||_2^2]$ coming from the random choice of graph $G$. Since our arguments are already quite lengthy, we do not pursue this generalization here, and refer the reader to \cite{lobel2010distributed, srivastava2011distributed} where similar arguments have been made. }

\section{Numerical simulations}\label{sec: numerical}
\subsection{Setup}
In this section, we simulate the RASGP algorithm on two classes of graphs, namely, random \as{directed} graphs and \as{bidirectional} cycle graphs. The main objective function is chosen to be a strongly convex and smooth Support Vector Machine (SVM), i.e. $F(\boldsymbol \omega, \gamma) = \frac{1}{2}\left(\Vert \boldsymbol \omega \Vert^2 + \gamma^2\right) + \as{C_N}\sum_{j=1}^N h(b_j(\bA_j^{\top} \boldsymbol \omega + \gamma ))$ where $\boldsymbol\omega\in \R^{d-1}$ and $\gamma\in \R$ are the optimization variables, and $\bA_j \in \R^{d-1}, b_j \in \{ -1, +1\}$, $j=1,\ldots,N$, are the data points and their labels, respectively. \as{The coefficient $C_N\in \R$ penalizes the points outside of the soft margin. We set $C_N = c/N, c=500$ in our simulations, which depends on the total number of data points.}  Here, $h:\R \rightarrow \R$ is the smoothed hinge loss, initially introduced in \cite{rennie2005loss}, defined as follows:
\begin{align*}
h(\xi)=\begin{cases}
-0.5 -\xi, \qquad &\text{if } \xi<0,\\
0.5(1-\xi)^2, \qquad &\text{if } 0\leq \xi<1,\\
0, \qquad &\text{if } 1 \leq \xi.
\end{cases}
\end{align*}

To solve this problem in a distributed way, we suppose all data points are spread among agents. Hence, the local objective functions are $f_i(\boldsymbol\omega_i,\gamma_i)= \frac{1}{2n}\left(\Vert \boldsymbol\omega \Vert^2 + \gamma^2\right) + \as{C_N}\sum_{j\in D_i} h(b_j(\bA_j^{\top}\boldsymbol\omega + \gamma ))$, where $D_i \subset \{1,2,\ldots,N\}$ is an index set for data points of agent $i$ and $N$ is the total number of data points. \as{We choose the size of the dataset for each local function to be a constant ($|D_i| = 50$), thus $N=50n$.} It is easy to check that each $f_i$ has Lipschitz gradients and is strongly convex with $\mu_i=1/n$.

\as{We will compare our results with a centralized gradient descent algorithm, which updates every $L_u$ iterations using the step-size sequence $\alpha_c(k) = L_u/(\mu k)$, in the direction of the \emph{sum} of the gradients of all agents.}

To make gradient estimates stochastic, we add a uniformly distributed noise $\be_i \sim \mathbb{U}[-b/2, b/2]^d$ to the gradient estimates of each agent and $\be_c \sim \mathbb{U} [-\sqrt{n}b/2, \sqrt{n}b/2]^d$ to the gradient of the centralized gradient descent, where $\mathbb{U}[b_1, b_2]^d$ denotes the uniform distribution of size $d$ over the interval $[b_1,b_2)$, $b_1<b_2$.
Note that $\be_i$ and $\be_c$ are bounded and have zero mean and $\E [\Vert \be_i \Vert^2] = db^2 /12$ and $\E [\Vert \be_c \Vert^2] = ndb^2 /12$.
We set $b=\as{4}$ for all simulations.

Agents wake up with probability $P_w$ and links fail with probability $P_f$, unless they reach their maximum allowed value where the algorithm forces the agent to wake up or the link to work successfully. The link delays are chosen uniformly between $1$ to $L_{del}$.

\as{Each dataset $D_i$ is synthetically generated by picking $25$ data points around each of the centers $(1,1)$ and $(3,3)$ with multivariate normal distributions, labeled $-1$ and $+1$, respectively.}
In generating strongly connected random graphs, we pick each edge with a probability of \as{$0.5$} and then check if the resulting graph is strongly connected; if it isn't, we repeat the process. 
\as{Since the initial step-sizes for the distributed algorithm can be very large (e.g., $\alpha(1) = 50$ for $n=50$), to stabilize the algorithms, both algorithms are started with $k_0 = 100$. This wouldn't affect the asymptotic convergence performance. Moreover, the initial point of the centralized algorithm and all agents in RASGP are chosen as $\mathbf{1}_d$.}

Let us denote by $\hat{\bz}(k):= (1/n)\sum_{i=1}^n \bz_i(k)$ the average of $\bz$-values of non-virtual agents. Then, we define \as{\textit{optimization errors}} $E_{dist}:= \Vert \hat{\bz}(k) - \bz^* \Vert^2$ and 
$E_{c}(k) :=  \Vert \bx_c(k) - \bz^* \Vert^2$ for RASGP and Centralized stochastic gradient descent, respectively.

Since our performance guarantees are for the expectation of (squared) errors, for each network setting, we perform \as{up to $1000$
Monte-Carlo simulations and use their corresponding performance to estimate the average behavior of the algorithms}.
\as{Since accurately estimating the \textit{true} expected value requires an extremely large number of simulations, in order to alleviate the effect of spikes and high variance, we take the following steps. First a batch of simulations are performed and their average  is calculated. Next, to obtain a smoother plot, an average over every $100$ iterations is taken. And finally, the median of these outputs over all the batches is our estimate of the expected value.}

We report two figures for each setting: one including the errors $E_{dist}$ and $E_c$, and another one including $k \times E_{dist}$ and $k\times E_c$ to demonstrate the convergence rates.

\as{Finally, to study the non-asymptotic behavior of RASGP and its dependence on network size $n$, we have compared the performance of the centralized stochastic gradient descent and RASGP over a bidirectional cycle graph, with error variances of $n^2\hat\sigma^2$ and $\sigma_i^2 = \hat \sigma^2$, respectively. Then, we plot the ratio $E_c(k)/E_{dist}(k)$ over $n$, for different iterations $k$.}

\subsection{Results}
Our simulation results are consistent with our theoretical claims (due to the performance of centralized and decentralized methods growing closer over time) and show the achievement of an asymptotic network-independent convergence rate.

Fig. \ref{fig: res1} shows that when there is no link failure or delay and all agents wake up at every iteration ($L_s=2$), RASGP and centralized gradient descent have very similar performance.
When we allow links to have delays and failures (see Fig. \ref{fig: res2}), \as{as well as asynchronous updates (see Fig. \ref{fig: res3})}, it takes longer for RASGP to reach its asymptotic convergence rate.

\begin{figure}
	\centering
	\begin{subfigure}[b]{0.45\textwidth}
		\includegraphics[width=\textwidth]{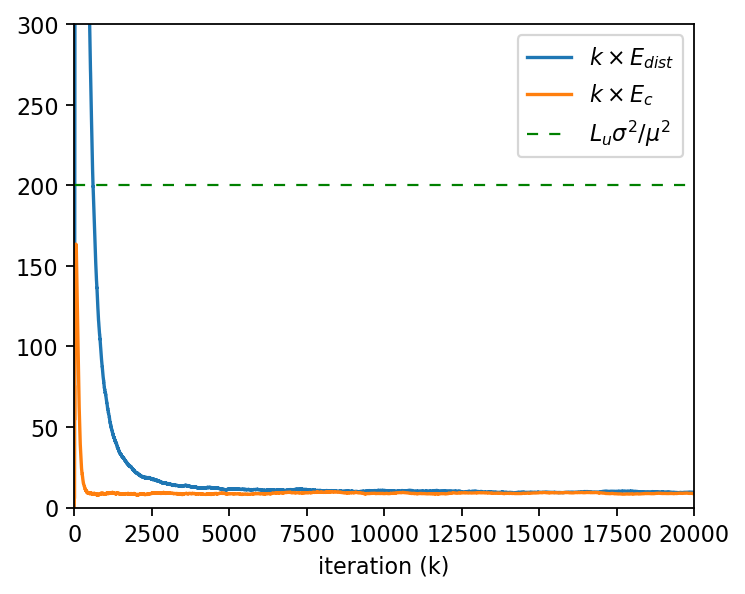}
		\caption{$k$ times squared errors.}
	\end{subfigure}
	\begin{subfigure}[b]{0.46\textwidth}
		\includegraphics[width=\textwidth]{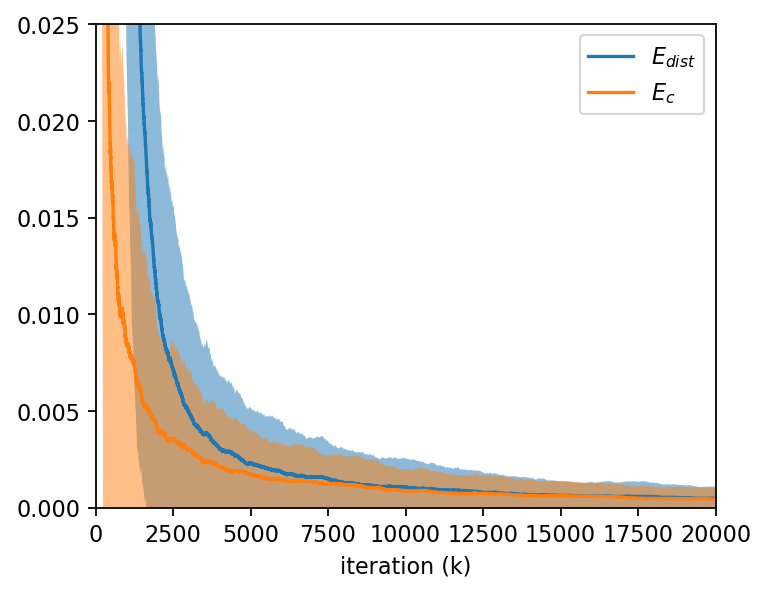}
		\caption{Errors and \as{ $1$-standard-deviation band}.}
	\end{subfigure}
	\caption{Results on a directed cycle graph \as{of size $n=50$}, synchronous with no delays and link failures ($P_w = 1$, $P_f=0$, $L_{del}=L_f =0, L_u=1$, $L_s = 2$).}
	\label{fig: res1}
\end{figure}

\begin{figure}
	\centering
	\begin{subfigure}[t]{0.45\textwidth}
		\includegraphics[width=\textwidth]{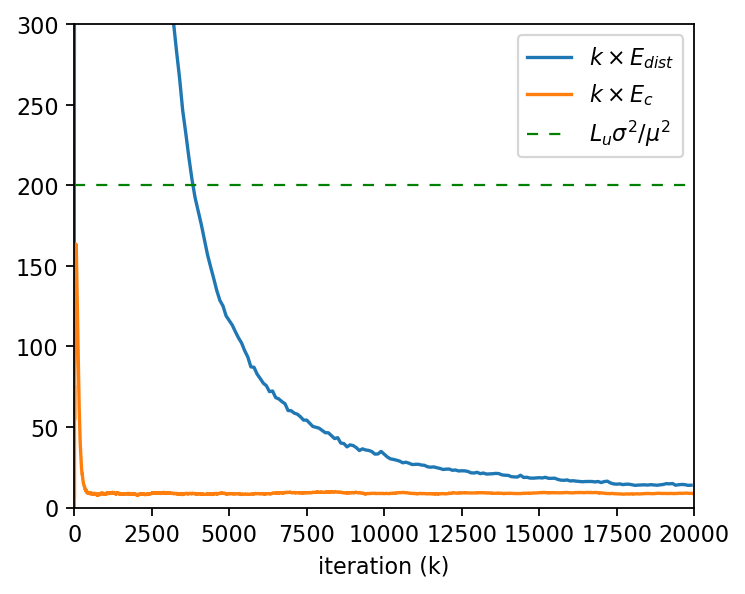}
		\caption{$k$ times squared errors.}
	\end{subfigure}
	\begin{subfigure}[t]{0.46\textwidth}
		\includegraphics[width=\textwidth]{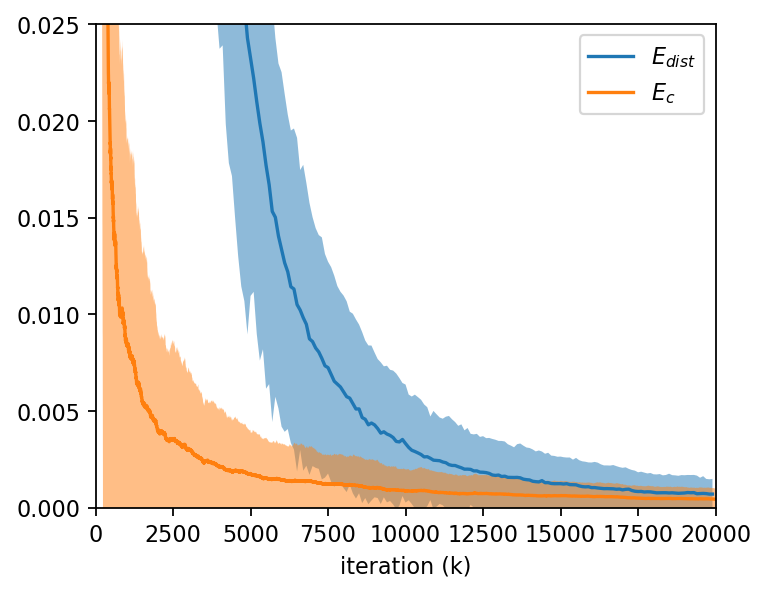}
		\caption{Errors \as{and $1$-standard-deviation band}.}
	\end{subfigure}
	\caption{Results on a directed cycle graph \as{of size $n=50$}, synchronous with delays and link failures ($P_w = 1$, $P_f=\as{0.3}$, $L_{del}=L_f =3, L_u=1$, $L_s = 7$).}
	\label{fig: res2}
\end{figure}

\begin{figure}[h]
	\centering
	\begin{subfigure}[b]{0.45\textwidth}
		\includegraphics[width=\textwidth]{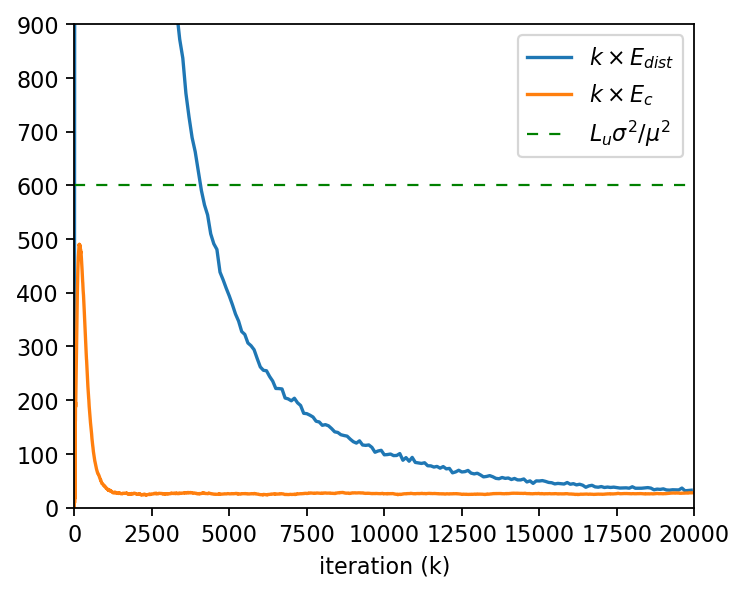}
		\caption{$k$ times squared errors.}
	\end{subfigure}
	\begin{subfigure}[b]{0.46\textwidth}
		\includegraphics[width=\textwidth]{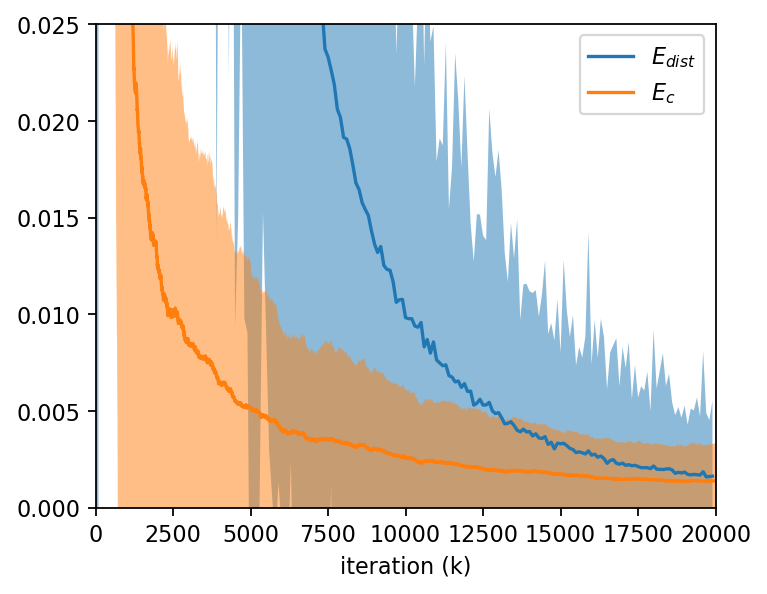}
		\caption{Errors \as{and $1$-standard-deviation band}.}
	\end{subfigure}
	\caption{Results on a directed cycle graph \as{of size $n=50$}, asynchronous with delays and link failures ($P_w = 0.5$, $P_f=\as{0.3}$, $L_{del}=L_f =3, L_u=3$, $L_s = 17$).}
	\label{fig: res3}
\end{figure}

We observe that, with all the other parameters fixed, the RASGP performs better on a random graph than on a cycle graph (see Figs. \ref{fig: res3} and \ref{fig: res4}). A possible reason is that the cycle graph has a higher diameter or mixing time compared to the random graph, resulting in a slower decay of the consensus error.
\begin{figure}
	\centering
	\begin{subfigure}[t]{0.45\textwidth}
		\includegraphics[width=\textwidth]{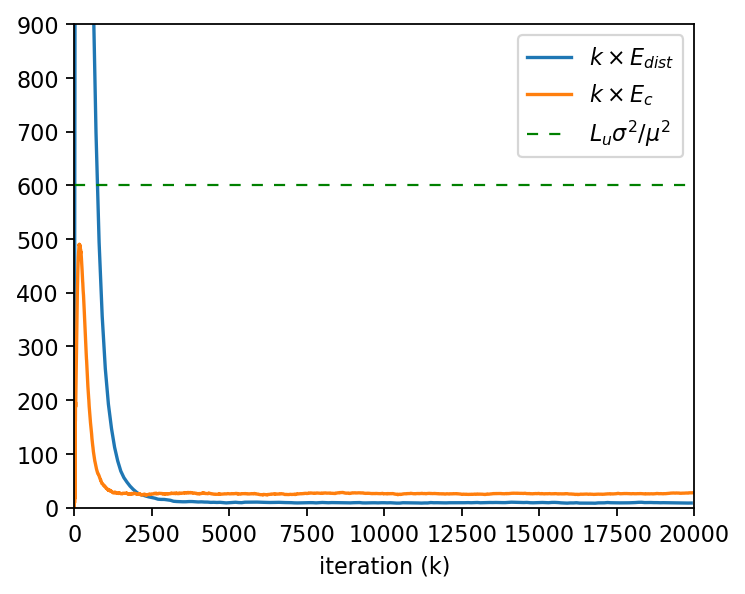}
		\caption{$k$ times squared errors.}
	\end{subfigure}
	\begin{subfigure}[t]{0.46\textwidth}
		\includegraphics[width=\textwidth]{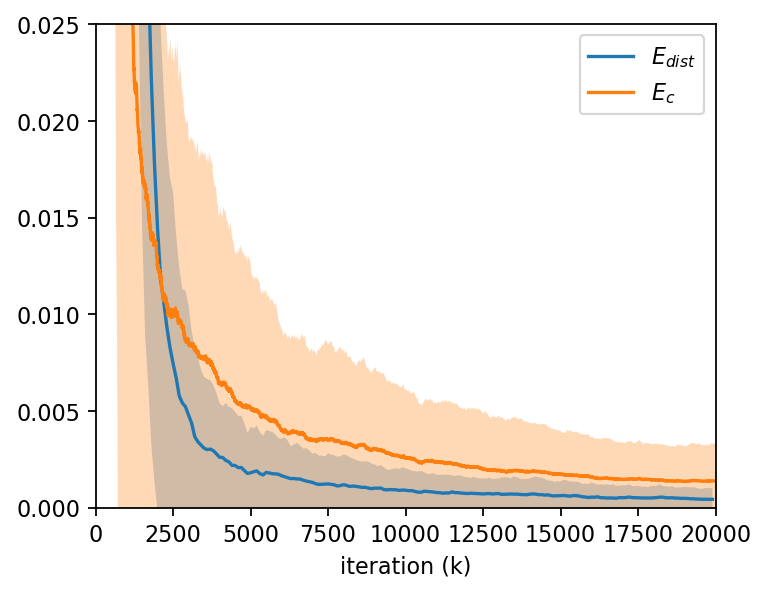}
		\caption{Errors \as{and $1$-standard-deviation band}.}
	\end{subfigure}
	\caption{Results on a directed random graph \as{of size $n=50$}, asynchronous with delays and link failures ($P_w = 0.5$, $P_f=\as{0.3}$, $L_{del}=L_f =3, L_u=3$, $L_s = 17$).}
	\label{fig: res4}
\end{figure}

\as{We notice that by fixing the network size, increasing the number of iterations brings us closer to linear speed-up (see Fig. \ref{fig: network size}). On the other hand, when fixing the number of iterations, increasing the number of nodes, after a certain point, does not help speeding up the optimization.
Moreover, by allowing link delays and failures (see Fig. \ref{fig: network size}(b)) we require more iterations to achieve network independence.}

\begin{figure}[h]
	\centering
	\begin{subfigure}[t]{0.45\textwidth}
		\includegraphics[width=\textwidth]{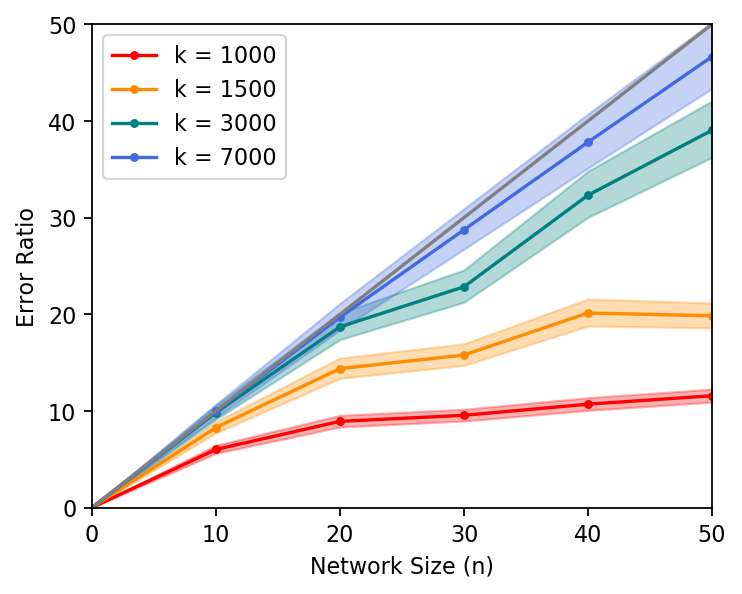}
		\caption{Synchronous with no delays and link failures.}
	\end{subfigure}
	\begin{subfigure}[t]{0.45\textwidth}
		\includegraphics[width=\textwidth]{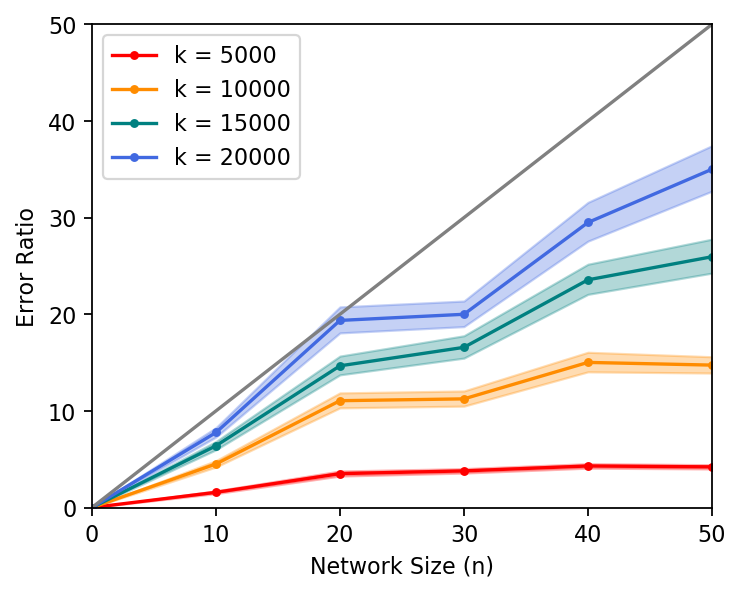}
		\caption{Synchronous with delays and link failure {($P_w = 1$, $P_f=0.3$, $L_{del}=L_f=3, L_u=1$, $L_s = 7$)}.}
	\end{subfigure}
	\caption{ Error ratio over network size. \as{Shaded areas correspond  to $1$-standard-deviation of the performance.}}
	\label{fig: network size}
\end{figure}

\section{Conclusions}\label{sec: conclusion}
The main result of this paper is to stablish asymptotically, network independent performance for a distributed stochastic optimization method over directed graphs with message losses, delays, and asynchronous updates. Our work raises several open questions. 

The most natural question raised by this paper concerns the size of the transients. How long must the nodes wait until the network-independent performance bound is achieved? The answer, of course, will depend on the network, but also on the number of nodes, the degree of asynchrony, and the delays. Understanding how this quantity scales is required before the algorithms presented in this work can be recommended to practitioners. 

More generally, it is interesting to ask which problems in distributed optimization can achieve network-independent performance, even asymptotically. For example, the usual bounds for distributed subgradient descent (see, e.g., \cite{nedic2018network}) depend on the spectral gap of the underlying network; various worst-case scalings with the number of nodes can be derived, and the final asymptotics are not network-independent. It is not immediately clear whether this is due to the analysis, or a fundamental limitation that will not be overcome.



\acks{The authors acknowledge support for this project by the AFOSR under grant FA9550-15-1-0394, by the ONR under grant N000014-16-1-224 and MURI N00014-19-1-2571, by the NSF under grants IIS-1914792, DMS-1664644, and CNS-1645681, and by the NIH under grant 1R01GM135930. A preliminary version of the results in Section \ref{sec: push-sum} has been published in the proceedings of the {\em American Control Conference} 2018 (\cite{olshevsky2018fully}).}


\newpage

\appendix
\section*{Appendix A. Proof of Lemma \ref{lem: rho u}}
\begin{proof}
	We use mathematical induction.
	For $k=0$ we have $x_{ij}^l(0)=0, \; \forall l$ and $u_{ij}^x(0)=\phi_i^x(0)=\rho_{ji}^x(0)=0$. By \eqref{eq: rho alg} and the definition of $u_{ij}^x$ and $x_{ij}^l$ we obtain,
	\begin{align*}
	\rho_{ji}^x(1) &=0,\\ 
	u_{ij}^x(1) &= (1-\sum_{l=1}^{L_d} \tau_{ij}^l(0))\phi_i^x(1),\\
	\sum_{l=1}^{L_d} x_{ij}^l(1) &= (\sum_{l=1}^{L_d} \tau_{ij}^l(0))\phi_i^x(1).
	\end{align*}
	Equation \eqref{eq: rho x1} is concluded from first equation above and \eqref{eq: u rho x} results by summing up all three equations above.
	
	Now assume this lemma is true for $k=0,\ldots,K-1$. We want to show it will be true for $k=K$ as well.
	In the following, $LHS$ and $RHS$ denote the left-hand-side and right-hand-side of \eqref{eq: rho x1} for $k=K$.
	By \eqref{eq: rho alg} we have,
	\begin{align*}
	LHS=\sum_{l=1}^{L_d} \tau_{ij}^l(K-l)[\phi_i^x(K+1-l)-\rho_{ji}^x(K)].
	\end{align*}
	Using \eqref{eq:x1} we obtain,
	\begin{align*}
	RHS= \sum_{l=1}^{L_d} \tau_{ij}^l(K-l)\mu_{ij}^x(K-l).
	\end{align*}
	Hence, it suffices to show that:
	\begin{align}\label{eq: equi1}
	\sum_{l=1}^{L_d} \tau_{ij}^l(K-l)[\phi_i^x(K+1-l) -
	\rho_{ji}^x(K)-\mu_{ij}^x(K-l)]=0.
	\end{align}
	By part (e) of Assumption \ref{asm: connvectivity}, at most one of the $\tau_{ij}^l(K-l)$, $l=1,\ldots,L_d$ is non-zero. If all are zeros, the result follows. 
	Now suppose $\tau_{ij}^l(K-l)=1$ for some $l$. Equation \eqref{eq: equi1} becomes, 
	\begin{align*}
	\phi_i^x(K+1-l)-\rho_{ji}^x(K)-\mu_{ij}^x(K-l)=0. 
	\end{align*}
	Plugging in the definition of $\mu_{ij}^x$, after rearrangement we obtain,
	\begin{align}\label{eq: equi3}
	\phi_i^x(K-l)-u_{ij}^x(K-l)=\rho_{ji}^x(K).
	\end{align}
	By the induction hypothesis, \eqref{eq: rho x1} holds for $k=K-t$, $t=1,\ldots,l$. Therefore,
	\begin{align*}
	\rho_{ji}^x(K+1-t)-\rho_{ji}^x(K-t)=x_{ij}^1(K-t).
	\end{align*}
	Hence,
	\begin{align*}
	\rho_{ji}^x(K)&=\rho_{ji}^x(K-l)+\sum_{t=1}^l(\rho_{ji}^x(K+1-t)-\rho_{ji}^x(K-t)) \\
	&=\rho_{ji}^x(K-l)+\sum_{t=1}^l x_{ij}^1(K-t)\\
	&=\rho_{ji}^x(K-l)+\sum_{l'=1}^l x_{ij}^{l'}(K-l)\qquad\text{(Lemma \ref{lem: x1=xl})}\\
	&=\rho_{ji}^x(K-l)+\sum_{l'=1}^d x_{ij}^{l'}(K-l).\qquad\text{(Lemma \ref{lem: xl>l=0})}
	\end{align*}
	Moreover, by the induction hypothesis, \eqref{eq: u rho x} holds for $k=K-l$, thus,
	\begin{align*}
	\phi_i^x(K-l)-u_{ij}^x(K-l)=\rho_{ji}^x(K-l)+\sum_{l'=1}^{L_d} x_{ij}^{l'}(K-l).
	\end{align*}
	Combining the two relations above we conclude \eqref{eq: equi3}.
	
	To show \eqref{eq: u rho x}, consider the following equations which are direct results of the definitions and \eqref{eq: rho x1} that we just showed for $k=K$:
	\begin{align*}
	u_{ij}^x(K+1) &= (1-\sum_{l=1}^{L_d} \tau_{ij}^l(K))\mu_{ij}^x(K),\\
	\rho_{ji}^x(K+1) &= \rho_{ji}^x(K) + x_{ij}^1(K),\\
	\sum_{l=1}^{L_d} x_{ij}^l(K+1) & = \sum_{l=2}^{L_d} x_{ij}^l(K) + \sum_{l=1}^{L_d} \tau_{ij}^l(K) \mu_{ij}^x(K).
	\end{align*}
	Summing up both sides of the equations above we have,
	\begin{align*}
	LHS &= u_{ij}^x(K+1)+ \rho_{ji}^x(K+1) + \sum_{l=1}^{L_d}x_{ij}^l (K+1),\\
	RHS &= \sum_{l=1}^{L_d} x_{ij}^l(K) + \rho_{ji}^x(K) + \mu_{ij}^x(K)\\
	& = \sum_{l=1}^{L_d} x_{ij}^l(K) + \rho_{ji}^x(K) + u_{ij}^x(K) - \phi_i^x(K) + \phi_i^x(K+1) = \phi_i^x(K+1).
	\end{align*}
	The last equality holds because of the induction hypothesis \eqref{eq: u rho x} for $k=K-1$, hence completing the proof.
\end{proof}


\vskip 0.2in
\bibliographystyle{natbib}
\bibliography{References}

\end{document}